\documentclass[reqno,11pt]{amsart}
\usepackage{amsmath, latexsym, stmaryrd, amsthm, array, caption,amssymb}
\usepackage[colorlinks=true, pdfstartview=FitV, linkcolor=blue, citecolor=blue, urlcolor=blue,pagebackref=false]{hyperref}
\usepackage[toc,page]{appendix}
\usepackage[utf8]{inputenc}
\usepackage{graphics}
\usepackage{pgfplots}
\usepackage{bm}
\usepackage{a4wide}
\usepackage{mathpazo}
\usetikzlibrary{arrows}
\usepackage[english]{babel}
\usepackage{enumitem}
  \numberwithin{equation}{section} 
\usepackage[T1]{fontenc}
\usepackage{mathtools}
\usepackage{bbm}
\usepackage[normalem]{ulem}

\newcommand{\N}{\mathbb N}
\newcommand{\R}{\mathbb R}
\newcommand{\Z}{\mathbb Z}
\newcommand{\E}{\mathbb E}
\newcommand{\T}{\mathbb T}
\renewcommand{\S}{\mathbb S}
\newcommand{\cE}{\mathcal E}

\renewcommand{\P}{\mathbb P}
\newcommand{\Path}{\mathcal{P}}

\newcommand{\1}[1]{\mathbbm{1}_{\left\{#1\right\}}}

\newcommand{\Var}{\mathrm{Var}}

\newcommand{\e}{{\rm e}}

\newcommand{\pp}{\mathrm P}
\newcommand{\ep}{\mathrm E}

\newcommand{\diam}{\mathrm{diam}}

\newcommand{\sG}{\overline{G}}
\newcommand{\sD}{ \overline{D}}
\newcommand{\spi}{ \overline \pi}
\newcommand{\tmix}{ t_{\rm mix}}

\newcommand{\jmax}{j_{\max}}
\newcommand{\imax}{i_{{}_{\max}}}
\newcommand{\jbar}{\overline j}
\newcommand{\jbarbar}{{\widehat j}}
\newcommand{\tGamma}{{\widetilde{\Gamma}}}
\newcommand{\tD}{{\widetilde{D}}}
\newcommand{\ttau}{{\tilde \tau}}
\newcommand{\talpha}{{\tilde \alpha}}
\newcommand{\tgamma}{{\tilde \gamma}}
\newcommand{\tW}{{\widetilde W}}
\newcommand{\Wmax}{   W_{{}_{\max}}}
\newcommand{\cR}{\mathcal R}

\def\be{\begin{eqnarray}}
\def\ee{\end{eqnarray}}
\def\ben{\begin{eqnarray*}}
\def\een{\end{eqnarray*}}

\newtheorem{proposition}{Proposition}[section]

\newtheorem{lemma}[proposition]{Lemma}

\newtheorem{theorem}[proposition]{Theorem}
\newtheorem{remark}[proposition]{Remark}
\newtheorem{observation}[proposition]{Observation}
\newtheorem{conjecture}[proposition]{Conjecture}

\newtheorem*{assumptionA2'}{Assumption A2'}

\newcounter{example}
\usepackage[comma, sort&compress]{natbib}

\begin{document}
\title[Scale-free percolation mixing time]{Scale-free percolation mixing time}

\author[A.\ Cipriani]{Alessandra Cipriani}
\address{{TU Delft} (DIAM), Mekelweg 4, 2628 CD, Delft, The Netherlands}

\email{a.cipriani@tudelft.nl}

\author[M.\ Salvi]{Michele Salvi}
\address{Dipartimento di Matematica, Universit\`a di Roma Tor Vergata, Via della ricerca scientifica 1, 00133, Rome, Italy}
\email{salvi@mat.uniroma2.it} 
\urladdr{\url{https://www.mat.uniroma2.it/~salvi/}}

\begin{abstract}
Assign to each vertex of the one-dimensional torus i.i.d. weights with a heavy-tail of index $\tau-1>0$. Connect then each couple of vertices with probability roughly proportional to the product of their weights and that decays polynomially with exponent $\alpha>0$ in their distance. The resulting graph is called scale-free percolation. The goal of this work is to study the mixing time of the simple random walk on this structure. We depict a rich phase diagram in $\alpha$ and $\tau$. 
	In particular we prove that the presence of hubs can speed up the mixing of the chain. We use different techniques for each phase, the most interesting of which is a bootstrap procedure to reduce the model from a phase where the degrees have bounded averages to a setting with unbounded averages.
	
	\noindent  \emph{AMS  subject classification (2010 MSC)}: 
	05C80 
	60K37, 
	05C82  	
	05C90  	
	
\end{abstract}

\maketitle


\section{Introduction}
In this article we study the mixing time of the simple random walk on a scale-free percolation random graph defined on the one-dimensional torus. 
\subsection{Spatial random graphs and scale-free percolation}
We say that a graph is {\it spatial} if its vertices occupy a position in a given metrical space. 
It is reasonable to believe that spatial random graphs are good candidates when one tries to model real-world networks where  agents have a geographical or physical position (commercial or social networks, telecommunications, brain cells...). This intuition has been confirmed by the fact that some of these models exhibit properties that are often observed in data: for example, it is often the case that the nodes of the network have a degree distribution with polynomial tails (the network is said to be {\it scale-free}), the nodes are separated by a relatively low number of edges (the network is said to be {\it small-world}) and they are highly clustered. While physicists have been studying this kind of networks for some time (see e.g.~the review \cite{B11}), mathematicians have begun to prove rigorous results on these models only in recent years. 
\smallskip

The scale-free percolation random graph (from now on SFP)  falls into the category of {\it inhomogeneous} spatial models: not only does the probability of linking two vertices of the graph depend on their position in space, but also on a random importance, or weight, that is assigned to each of them. 
SFP can be considered a combination of long-range percolation (a random graph where the probability of linking two nodes decays, roughly, polynomially in their distance, ~\cite{S83}) and inhomogeneous random graphs, such as the Norros-Reittu model (where nodes with a high weight are likelier to be linked, \cite{NR06}). Since its introduction in \cite{DVH13} SFP has been the object of intense research. In particular, the scale-free property has been proved for the discrete-space model, where the nodes lay on $\Z^d$, in \cite{DVH13}. In its continuum counterpart, where the position of the nodes is given by a Poisson point process, the graph has been shown to be scale-free in an annealed sense in \cite{DW18} and in a quenched sense in \cite{DS21}. The convergence in distribution for the maximum of the degrees on a growing observation window has been studied in \cite{BS19}. The problem of graph distances in SFP has been addressed in the original paper for discrete-space and in \cite{DW18} for the continuum. Depending on the parameters of the model, SFP can exhibit the small-world property (the graph distance between vertices behaves asymptotically as the $\log$ of their Euclidean distance), the ultra-small-world property (the graph distance is $\log\log$ of the Euclidean distance) or can be comparable to Euclidean distances. A good deal of effort is being put in finding the precise order of these distances, see \cite{DHW15, HK17, HH21}. Finally, in \cite{DS21} the authors show the positivity of the clustering coefficient for continuous SFP. 

\smallskip

Among the other few spatial models for which these properties have been proved to various extents, we mention the ultra-small scale-free geometric network~\citep{Y06}, the hyperbolic
random graph~\citep{GPP12}, the spatial preferential attachment model~\citep{JM15}, the age-dependent random connection model~\citep{GGLM19} and the geometric inhomogeneous random graph~\citep{BKL19}.
As noted in \cite{GHMM19}, most of these models can be thought as particular cases of the more general weight-dependent random connection model. To further confirm our original motivation, we point out that 
some of these random graphs have been proposed to model real-world networks such as the Internet \citep{PKBV10}, banking systems \citep{DHW15} and livestock trades \citep{DS21}. 

\subsection{Stochastic processes on spatial inhomogeneous random graphs}
Motivated by applications (the spread of fake news on social media, the outbreak of an epidemics, the diffusion of a computer virus...), we take a step further and look at stochastic processes that evolve {\it over} inhomogeneous spatial networks. 	
The interest of the mathematical community on this topic is quite recent and a few references are available. Among them we find \cite{CF16} and \cite{KL16}, where the authors study bootstrap percolation on the hyperbolic random graph and on the geometric inhomogeneous random graph respectively; \cite{KL20}, dealing with first passage percolation on different graph models; \cite{JM17}, about a push\&pull protocol on the spatial preferential attachment.

\smallskip

One of the most basic and studied processes on a graph is the simple random walk, where at each time-step a particle moves from its current location to any neighboring vertex with equal probability.
As far as we could check, the only available results for the simple random walk on spatial inhomogeneous random graphs  are the analysis of 
transience and recurrence for weight-dependent random connection models in \cite{GHMM19} and
for SFP in \cite{HHJ17}. 
The {\it mixing time} of the simple random walk is, roughly put, the time needed for the distribution of the chain to approach its invariant measure. 
Quantifying the mixing time of a Markov chain is
of primary importance, for example, for its connection with the spectral gap of the chain (see~Remark \ref{yri}) and in computer science for sampling through Monte Carlo procedures.
We refer to \cite{LP17} for a complete account on the subject. %
We could find very few works where the random walk mixing time is studied for models that go beyond lattices or graphs without an underlying geometry. \cite{BBY08} and \cite{CS12} deal with long-range percolation in dimension $d=1$ and $d\geq 2$ respectively, while \cite{DGGJV20} analyzes another closely related model. 

\subsection{Our contribution}
In the present paper we address the problem of finding the order of the mixing time for the simple random walk on an SFP constructed on the one-dimensional torus of size $N$. To our knowledge, this is the first time the mixing time is analyzed for a walk on an {\it inhomogeneous} spatial random graph.
 The graph, called $G_N$, is built as follows: to each node in $\T_N:=\{1,\dots,N\}$ we assign independent weights $(W_x)_{x\in\T_N}$ following a Pareto distribution of parameter $\tau-1$, with $\tau>1$. Once we have fixed the weights, we add an edge between node $x$ and node $y$ with probability 
\begin{align*}
1-\exp\{-W_xW_y\|x-y\|^{-\alpha}\}
\end{align*}
where $\|\cdot\|$ is the torus-distance and 
where $\alpha>0$ is the parameter which tunes the influence of the distance between the nodes over the linking probability. So $\alpha$ and $\tau$ are the two parameters of the model, and we call  $\gamma:=\alpha(\tau-1)$. It is possible to show that the degrees of the nodes have a heavy tail of parameter $\gamma$, see \cite{DVH13}. We consider then a lazy simple random walk on $G_N$ and study its mixing time $\tmix(G_N)$, see Section \ref{orizzonti} for a precise definition.

\smallskip

We are inspired by \cite{BBY08}, who studied the same problem for the long-range percolation random graph on $\T_N$. Long-range percolation is equivalent to a version of SFP where all the weights are set equal to $1$, which morally corresponds to the case $\tau=\infty$. The authors of \cite{BBY08} prove that the simple random walk on long-range percolation undergoes a phase transition in the parameter $\alpha$: 
when $1<\alpha<2$, the mixing time is of order
$N^{\alpha-1}$, whereas for $\alpha>2$ it is of order $N^2$ (all up to polylogarithmic factors). Note in particular the remarkable discontinuity of the exponent at $\alpha=2$.

\smallskip

For SFP we depict an almost complete phase diagram in $\alpha$ and $\tau$ with a rich variety of phases. Up to correcting factors, we show that $\tmix(G_N)$ behaves as follows, cfr. Figure \ref{fig:fd_mix}:  

\smallskip

\begin{enumerate}[label=(\roman*),ref=(\roman*)]
	\item \label{case1} for $\gamma<1$, $\tmix(G_N)$ is upper bounded by a power of $\log N$ (Theorem \ref{gammino}).
	
	\smallskip
	
	\item \label{case2} For $1<\gamma<2$  and $\tau<2$, $\tmix(G_N)$ is of order $N^{\gamma-1}$ (Theorem \ref{thm:gamma_tra_12}).
	
	\smallskip
	
	\item \label{case3} For $\alpha\in(1,2)$ and $\tau>2$, $\tmix(G_N)$ is at least of order $N^{\alpha-1}$ (Theorem \ref{strn}).
	
	\smallskip
	
	\item \label{case4} For $\alpha>2$ and $\gamma>2$, $\tmix(G_N)$ is of order $N^2$ (Theorem \ref{brng}).
\end{enumerate}
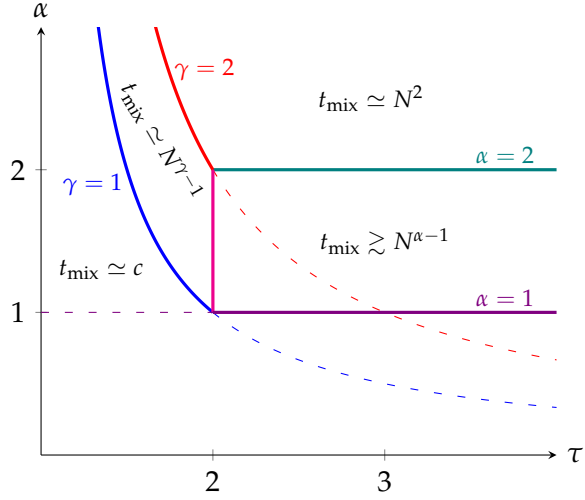
\begin{figure}[ht!]
	\begin{center}
		\begin{tikzpicture}
		\begin{axis}[
		axis lines = left,
		xlabel = $\tau$,
		ylabel = {$\alpha$},
		every axis x label/.style={	at={(ticklabel* cs:1.00)},anchor=west,},
		every axis y label/.style={ at={(ticklabel* cs:1.00)}, anchor=south,},
		xmin=0, xmax=3, ymin=0, ymax=3, xtick={1,2}, ytick={1,2}, xticklabels={$2$,$3$}
		]
		
		\addplot [domain=1:3, samples=100, color=red,style={loosely dashed}]{2/x};
		\addplot [domain=0:1, samples=100, color=blue,style=very thick]{1/x};
		\addplot [domain=1:3, samples=100, color=blue,style=loosely dashed,]{1/x};
		\addplot [domain=1:3, samples=100, teal, very thick]{2};
		\addplot[samples=20, magenta, very thick] coordinates {(1,1)(1,2)};
		\addplot [domain=0:1, samples=100, color=red,style=very thick,smooth]{2/x};
		\addplot [domain=1:3, samples=100, color=violet,style=very thick,]{1};
		\addplot [domain=0:1, samples=10, color=violet,style=loosely dashed,]{1};
		\node[] at (axis cs: 0.3,1.9) {\color{blue}\footnotesize $\gamma=1$};
		\node[] at (axis cs: 0.95,2.7) {\color{red}\footnotesize $\gamma=2$};
		\node[] at (axis cs: 2.7,2.1) {\color{teal}\footnotesize $\alpha=2$};
		\node[] at (axis cs: 2.7,1.1) {\color{violet}\footnotesize $\alpha=1$};  
		\node[] at (axis cs: 1.92,2.5) {\footnotesize $t_{\rm mix}\simeq N^2$};
		\node[] at (axis cs: 2,1.5) {\footnotesize $t_{\rm mix}\gtrsim N^{\alpha-1}$};
		\node[] at (axis cs: 0.35,1.3) {\footnotesize $t_{\rm mix}\simeq c$};
		\node[rotate=300] at (axis cs: 0.70,2.2) {\footnotesize $t_{\rm mix}\simeq N^{\gamma-1}$};
		\end{axis}
		\end{tikzpicture}
	\end{center}
	\caption{Phase diagram of the mixing time of the simple random walk on the SFP random graph on the one-dimensional torus of size $N$. The symbols $\gtrsim$~and~$\simeq$~indicate an (in)equality up to a slowly varying function in $N$.}	\label{fig:fd_mix}
\end{figure}

Let us further comment on these results also comparing them with the phase-diagram borrowed from \cite{HHJ17}, see Figure~\ref{mrvjs}. In this diagram, we report the asymptotic behaviour of the graph distances between vertices with respect to their Euclidean distance for SFP on $\Z^d$ (so we are only interested in the case $d=1$).   

{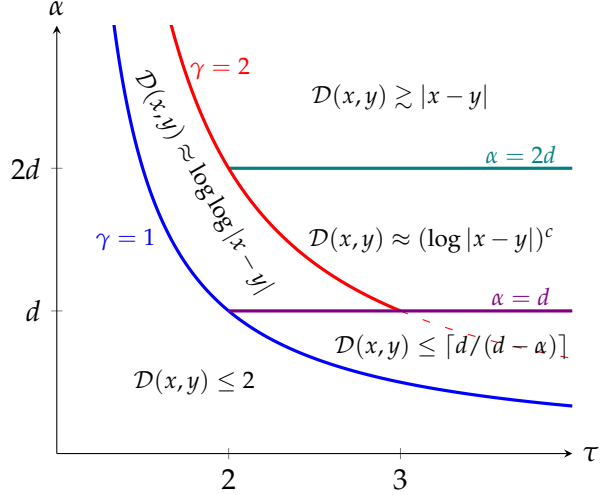
\begin{figure}[ht!]
		\centering
		\begin{tikzpicture}
		\begin{axis}[
		axis lines = left,
		xlabel = $\tau$,
		ylabel = {$\alpha$},
		every axis x label/.style={	at={(ticklabel* cs:1.00)},anchor=west,},
		every axis y label/.style={ at={(ticklabel* cs:1.00)}, anchor=south,},
		xmin=0, xmax=3, ymin=0, ymax=3, xtick={1,2}, ytick={1,2}, xticklabels={$2$,$3$}, yticklabels={$d$,$2d$},
		]
		
		\addplot [domain=0:1, samples=100, color=blue,style=very thick]{1/x};
		\addplot [domain=1:3, samples=100, color=blue,style=very thick]{1/x};
		\addplot [domain=1:3, samples=100, color=teal, style=very thick]{2};
		\addplot [domain=0:2, samples=100, color=red,style=very thick,]{2/x};
		\addplot [domain=2:3, samples=100, color=red,style={loosely dashed}]{2/x};
		\addplot [domain=1:3, samples=100, color=violet,style=very thick,]{1};
		\node[] at (axis cs: 0.4,1.5) {\color{blue}\footnotesize $\gamma=1$};
		\node[] at (axis cs: 0.95,2.7) {\color{red}\footnotesize $\gamma=2$};
		\node[] at (axis cs: 2.7,2.1) {\color{teal}\footnotesize $\alpha=2d$};
		\node[] at (axis cs: 2.7,1.1) {\color{violet}\footnotesize $\alpha=d$};  
		\node[] at (axis cs: 2,2.5) {\footnotesize $\mathcal D(x,y)\gtrsim  |x-y |$};
		\node[] at (axis cs: 2.18,1.5) {\footnotesize $\mathcal D(x,y)\approx (\log |x-y |)^c$};
		\node[] at (axis cs: 0.8,0.5) {\footnotesize $\mathcal D(x,y)\leq 2$};
		\node[] at (axis cs: 2.3,0.75) {\footnotesize $\mathcal D(x,y)\leq \lceil d/(d-\alpha)\rceil$};
		\node[rotate=300] at (axis cs: 0.85,1.9) {\footnotesize $\mathcal D(x,y)\approx \log\log |x-y|$};
		\end{axis}
		\end{tikzpicture}
		\caption{Phase diagram of graph distances in SFP on the infinite lattice $\Z^d$ from \cite{HHJ17}. $\mathcal D(x,y)$ indicates the graph distance between points $x$ and $y$ (conditioned on being in the infinite cluster of the graph, see the paper for more details), while {$|\cdot |$} is the Euclidean distance. See \cite{HH21} for the precise meaning of the symbol $\approx$.}
				\label{mrvjs}
	\end{figure} 
\ref{case1}: in this case, the mean degree of the nodes of SFP on $\T_N$ goes to infinity as $N$ grows. Nevertheless, the resulting graph is far from being the complete graph (even if the graph distances on the infinite lattice for $\gamma<1$ are bounded by 2, see Figure \ref{mrvjs}) and the bound of $\tmix(G_N)$ in this regime presents several challenges, see Section \ref{sec:techniques}. Our result here holds almost surely and we point out that this fact becomes fundamental for the proof of \ref{case2}. When $\alpha<1$ and $\gamma>1$ the mean degree is also unbounded and $\tmix(G_N)$ should be again polylogarithmic.
We do not study this regime since it is irrelevant for the investigation of the polynomial mixing of the other cases.

\smallskip

\ref{case2}: we consider this to be the most interesting regime, both from a mathematical viewpoint and for applications (see \cite{SHL19}). The degrees have bounded first moment in $N$, but unbounded variance, while the weights have infinite mean. The statement on $\tmix(G_N)$ holds in probability, see Theorem \ref{thm:gamma_tra_12}, and its proof consists in a bootstrap procedure that 
brings us back to the model with $\gamma<1$, see Section \ref{sec:techniques} for more details. The exponent $\gamma-1=\alpha(\tau-1)-1$ of the mixing time confirms the intuitive fact that the presence of nodes with a very high degree (due to small values of $\tau$), also called hubs, speeds up the mixing (note that, conversely, there are  dynamics that are slowed down by hubs, see e.g.~\cite{JM17}). This is a fundamental novelty with respect to the long-range percolation of \cite{BBY08}, where this phase clearly does not appear.
Another interesting point is that the graph with $1<\gamma<2$ and $\tau<2$ has a very small diameter (cfr.~Proposition \ref{diametro}), but a polynomial mixing time. In this regime the graph distances in the infinite lattice behave even as the $\log\log$ of the Euclidean distances, see Figure~\ref{mrvjs}: the graph exhibits the {ultra-small world} property.

\smallskip

\ref{case3}: in this case we are only able to show a lower bound of order $N^{\alpha-1}$. We make the conjecture that an upper bound of the same order applies:
\begin{conjecture}
	Let $1<\alpha<2$ and $\tau>2$. There exist  $c>0$ such that
	\begin{align*}
	\P(\tmix(G_N)<N^{\alpha-1}(\log N)^{c})\xrightarrow{N\to\infty}1\,.
	\end{align*}
\end{conjecture}
There are several reasons to believe this statement to be true. First of all, when $\tau>2$ the weights have a finite mean, so it is reasonable to think that the model would behave like long-range percolation and the order of $\tmix(G_N)$ would match the $N^{\alpha-1}$ of \cite{BBY08}. Secondly, we can prove the conjecture in certain sub-regions of this phase.
Thirdly, the probability of linking two distant nodes undergoes a phase-transition at $\tau=2$, cfr.~Lemma \ref{bosh}.
One might object, by looking at Figure \ref{mrvjs}, that the ultra-small world regime extends to the ``triangle'' $\tau>2,\,\alpha>1,\,\gamma<2$ (the area between the pink, the purple and the red dotted lines in Figure \ref{fig:fd_mix}), so one might think that the order $N^{\gamma-1}$ could extend at least to that part of the diagram. Indeed, our proof for the upper bound of \ref{case2} works also in the triangle, yielding an upper bound of $N^{\gamma-1}$. Nevertheless, we believe this bound to be suboptimal (notice that for $\tau>2$ one has that $\gamma-1>\alpha-1$). This is because we believe $\tmix(G_N)$ to be an increasing function on $\tau$ (that is, higher weights bring to a faster mixing), and if $N^{\gamma-1}$ was the right order in the triangle and $N^{\alpha-1}$ outside of the triangle, we would have a sudden decrease of the order of $\tmix(G_N)$ when increasing $\tau$ in correspondence of $\gamma=2$ (the red dotted line in Figure~\ref{fig:fd_mix}).
 
\smallskip 

\ref{case4}: the last regime has the slowest possible mixing, $N^2$. 
This could be expected by noticing in Figure~\ref{mrvjs} that the graph distances between points behave linearly in their Euclidean distance. 

}

\subsection{Techniques and outline of the paper}\label{sec:techniques}
We introduce precisely our model and state our main theorems in Section~\ref{modelandresults}. After giving some preliminary results in Section~\ref{preliminary}, we carry out the proofs. We summarize here the key ideas and techniques we use.
\begin{itemize}[leftmargin=*]
\item Upper bound of \ref{case1}, Section \ref{sec:gammino}:
the main idea is to study first a simplified model where the torus geometry is ignored: we introduce a new random graph $\sG_N$ where two nodes are linked if and only if the product of their weights is larger than $N^\alpha(\log N)^2$. We show that, if the weights follow independent Pareto distributions of parameter $\tau-1$, then  the mixing time of the simple random walk on $\sG_N$ is polylogarithmic in Proposition \ref{gmnsmplif}. In order to do so, we study the Cheeger constant of  $\sG_N$. While the Cheeger constant is usually used to find lower bounds on the mixing thanks to a test set, here we analyze the  bottleneck ratio of {\it all} the possible sets of vertices. This requires a thorough slicing of the  set of vertices according to their weights and then concentration inequalities to control the number of nodes and their degree in each slice (Proposition \ref{prop:VjConcentra}). Once the mixing of the toy model $\tmix(\sG_N)$ has been established, we prove that it is substantially equivalent to $\tmix(G_N)$, see Proposition \ref{Ub_con_ridotto}.

\smallskip

\item Upper bound of case \ref{case2}, Sections \ref{sec:importante} and \ref{sec:importante2}: this bound requires the most elaborate ideas. The initial approach is inspired by \cite{BBY08}: we divide the torus $\T_N$ into $K$ chunks $S_1,\dots,S_K$ of length $L=N^{\gamma-1+\varepsilon}$ for some small $\varepsilon>0$. We collapse all the points of the $S_j$'s into a unique point and we obtain a new graph $\Gamma$ on the torus $\T_K$. While in \cite{BBY08} the  graph resulting from a similar operation stochastically dominates  an Erd\H os-R\'enyi random graph with link probability $\log K/K$, we end up we something quite different. By a rescaling of order $L^{-(\tau-1)^{-1}}$ of the weights and a coupling procedure, we can show that $\Gamma$ stochastically dominates a random graph $\tGamma$ which is {\it again} an SFP. The fundamental point is that this time $\tGamma$ has parameters $\talpha<\alpha$ and $\ttau=\tau$ such that $\tgamma=\talpha(\ttau-1)<1$. Furthermore, by using  multicommodity flows as a tool, it is possible to bound  $\tmix(G_N)$ by some quantities related to $G_N$ and $\tGamma$, see Lemma \ref{keylemma}. This is a refinement of the approach of \cite{BBY08}, since for long-range percolation much cruder bounds are sufficient.

The second part of the proof in Section \ref{sec:importante2} is devoted to the estimate of these quantities. The first one is $\tmix(\tGamma)$, which we already know to be at most polylogarithmic in $K$ by case \ref{case1}. The second one is the largest diameter $\Delta_{G_N}$ of the graphs induced by $G_N$ on each $S_j$; we bound $\Delta_{G_N}$ by adapting an argument of \cite{DVH13} for graph distances on $\Z^d$. The third is the ratio of the total number of edges in $G_N$ and in $\tGamma$, which we prove to be very stable thanks to a Bernstein-type concentration inequality in Proposition \ref{ratiodegrees}. The last two quantities, called $\Pi_{G_N}$ and $R_{G_N,\tGamma}$ in Lemma \ref{keylemma}, involve the equilibrium measure on $G_N$ and its relation with the equilibrium measure on $\tGamma$. Their study, Propositions \ref{ratiopis} and \ref{ags}, is quite involved and technical. One is forced not only to bound the maximum of the degrees on each $S_j$, but also its product with the sum of all the other degrees in $S_j$. We achieve an optimal bound by  using several times the Fuk--Nagaev inequality, see Theorem \ref{thm:fuknagaev}.

\smallskip

\item Upper bound of case \ref{case4}, Section \ref{tbls}: we use a second moment method to show the concentration of the total degree of $G_N$ around its mean, see Lemma \ref{conc_Dtot}. In turn, this allows us to easily bound $\tmix(G_N)$ via the hitting times of the chain in Proposition \ref{prop:gammaalphagrande}.

\smallskip

\item Lower bounds of all regimes, 
Section \ref{lowerbounds}: the lower bounds of cases \ref{case2} and \ref{case3} are obtained at once in Proposition \ref{boundary_boundal} using a test set in the Cheeger constant. The lower bound of case \ref{case4} uses the parallel between random walks on graphs and electrical networks: similarly to \cite{benjaberger} we show that there exists a positive fraction of nodes that are cut-points for the graph (in a particular sense, see Lemma \ref{lem:good_cut}) and infer in Proposition \ref{prop:LBag22} that the mixing must be at least the square of the number of vertices.
\end{itemize}

\section{Model and results}\label{modelandresults}

\subsection{Scale-free percolation on the torus}\label{shesshion}
We describe now in detail the distribution of the SFP random graph $G_N$.  The set of vertices of $G_N$ is $\{1,2,...,N\}$, which we identify with $\T_N:=\Z/N\Z$, the torus of size $N$.
The edge set $E(G_N)$ with law $\P$
is constructed in two steps:
\begin{itemize}[leftmargin=*]
	\item for $\tau>1$, we associate to each $x\in \T_N$ a random weight $W_x$ such that the weights under $\P$ are independent and follow a Pareto distribution with parameter $\tau-1$, that is 
	\begin{align}\label{weightdistribution}
	\P(W_1 \geq t)=\begin{cases}
	           t^{-{(\tau-1)}}& t\ge 1\\
	           1 & t<1\,.
	          \end{cases}
	\end{align}
	\item Once we have fixed the weights of the nodes of the graph, we connect independently any couple of nodes $x,\,y\in \T_N$ with probability
	\begin{align}\label{eq:con_proba}
	\P(x\leftrightarrow y\,|\,W_x\,,W_y):=1-\e^{-\frac{W_xW_y}{\|x-y\|^\alpha}}
	\end{align}
	where $\alpha>0$ is another parameter of the graph and $\|x-y\|$ denotes the distance of $x$ and $y$ on the torus, that is, $\|x-y\|:= |x-y|\wedge (N-|x-y|)$, where $\wedge$ indicates the minimum between the two. 
\end{itemize}
To make sure that $G_N$ is connected, we will also impose that $x$ and $ x+1$ are linked for all $x\in \T_N$ (we identify $N+1$ with $1$). 

\begin{remark}
	The Pareto tail for the weights in \eqref{weightdistribution} has been chosen for convenience, rather than some more general distribution as in \cite{DVH13}. We preferred to sacrifice generality for cleaner and more readable calculations in the proofs. We believe that our results would substantially remain true for weights whose distribution has a regularly varying tail of index $\tau-1$.
\end{remark}

For a given graph $G=(V,E)$ we write $\{x\stackrel{G }{\leftrightarrow} y\}$, or simply $\{x\leftrightarrow y\}$ when there is no risk of confusion, for the event that $x$ and $y$ are connected by an edge. 
$D_x=D_x(G):=\sum_{y\neq x}\1{y \leftrightarrow x}$  indicates the degree of node $x\in V$. For a set $A\subseteq V$, we write $D_A=D_A(G):=\sum_{x\in A}D_x$, so that  $D_G=2|E|$ is twice the total number of edges. For two sets $A,B\subseteq V$, we also let $D_{A,B}:=\sum_{x\in A,\,y\in B}\1{x\leftrightarrow y}$ be the number of edges going from $A$ to $B$. The diameter of $G$ is defined as
\[
 \diam(G):=\max_{x,\,y\in V}\mathcal D(x,\,y)\, ,
\]
where $\mathcal D(x,\,y)$ denotes the graph distance between points $x$ and $y$, that is, the minimal number of edges of the graph one has to cross to go from $x$ to $y$.

\subsection{The simple random walk on $G_N$ and its mixing time}\label{orizzonti}
For a given realization of the graph $G_N$, we define now the simple random walk $(X_n)_{n\in\N_0}$. This is the Markov chain with transition matrix given by
\begin{equation}\label{transitionmatrix}
\begin{cases}
P(x,y)&= \frac{1}{2D_x}\1{y\leftrightarrow x}\\
P(x,x)&=\frac 12
\end{cases}.
\end{equation}
We consider this {\it lazy} version of the walk in order to avoid periodicity issues. The invariant (in fact, reversible) measure $\pi$ for the walk $(X_n)_{n\in\N_0}$  is defined as 
\begin{align}\label{invariant}
\pi(x)=\pi_{G_N}(x):=\frac{D_x}{D_{G_N}},\quad x\in \T_N\,.
\end{align}
Laziness ensures that, no matter the starting point of the walk, the distribution of $X_n$ will approach  $\pi$
by the ergodic theorem. Our goal is to quantify how long we will have to wait before these two measures are in some sense close to each other. To this end, we define 
\begin{align*}
d(n):=\sup_{x\in \T_N}\|P^n(x,\cdot)-\pi(\cdot)\|_{ {}_{\rm TV}},\quad n\in\N
\end{align*}
as the distance between the distribution of the walk on $G_N$ at time $n$ and $\pi$ when starting from the worst possible vertex. Recall that the total variation distance for two measures $\mu$ and $\nu$ on $G_N$ is given by
\begin{align*}
\| \mu-\nu\|_{ {}_{\rm TV}}:=\frac 12\sum_{x\in \T_N}|\mu(x)-\nu(x)|\,.
\end{align*}
The time for $(X_n)_{n\in\N_0}$  to get close to $\pi$ is the so-called mixing time of the chain:
\begin{align*}
\tmix(G_N):=\inf\big\{n\in\N:\,d(n)<\tfrac 14\big\}\,.
\end{align*}
For other graphs $G$ we will write $\tmix(G)$ for the analogous quantity. Notice that the quantity $1/4$ in the definition is arbitrary, see ~\citet[Section~4.5]{levinperes}.

\smallskip

\subsection{Main results}\label{mainresults}
We call
\begin{align*}
\gamma:=\alpha(\tau-1)\,.
\end{align*}
The main results of our work are the following.

\begin{theorem}\label{gammino}
	Let $\gamma<1$. There exists $c>0$ such that, $\P\mbox{\,--\,a.s.}$, for $N$ large enough,
	\begin{align*}
		\tmix(G_N)\le (\log N)^c \,.
	\end{align*}
\end{theorem}

\smallskip

Recall that a measurable function $\ell:\,\N\to(0,\infty)$ is said to be slowly varying if, for all $s\in(0,\infty)$, it holds that $\lim_{t\to\infty}\ell(st)/\ell(t)=1$~\cite[Section~1.2]{BGTT89}.
\begin{theorem}\label{thm:gamma_tra_12}
	Let $1<\gamma<2$ and $1<\tau<2$. There exist a constant $c>0$ and a slowly varying function $\ell:\,\N\to(0,\infty)$ such that
	\begin{align*}
		\P\big(N^{\gamma-1}(\log N)^{c}
		\leq \tmix(G_N)\leq
		N^{\gamma-1}\ell(N)\big)\xrightarrow{N\to\infty}1\,.
	\end{align*}
\end{theorem}

\smallskip

\begin{theorem}\label{strn}
	Let $1<\alpha<2$ and $\tau>2$. There exists  $c>0$ such that, $\P\mbox{\,--\,a.s.}$, for $N$ large enough,
	\begin{align*}
	\tmix(G_N)
		\geq N^{\alpha-1}(\log N)^{c}\,.
	\end{align*}
\end{theorem}

\smallskip

\begin{theorem}\label{brng}
	Let $\alpha>2$ and $\gamma>2$. There exist  $c_1,c_2>0$ such that, $\P\mbox{\,--\,a.s.}$, for $N$ large enough,
	\begin{align*}
	c_1 N^{2}
	\leq \tmix(G_N) 
	\leq N^{2}(\log N)^{c_2}\,.
	\end{align*}
\end{theorem}
 The results of these theorems are summarized in the phase-diagram of Figure~\ref{fig:fd_mix}, where we use the symbols $\lesssim$ and $\gtrsim$ to omit corrections with slowly varying functions, see Section \ref{sec:not}.
\begin{remark}\label{yri}
We point out that there is a close relation between $\tmix(G_N)$ and the spectral gap of the chain. Recall that, letting $1=\lambda_1(N)>\,\lambda_2(N)\geq\dots\geq\lambda_N(N)$ be the eigenvalues of the matrix $P=(P(x,y))_{x,y\in\T_N}$, the spectral gap is defined as $1-\lambda_2(N)$. Then (see for example \citet[Theorems 12.3 and 12.4]{LP17}) it holds 
	\begin{align*}
	\Big(\frac{1}{1-\lambda_2(N)}-1\Big)\log 2
		\leq\tmix(G_N)
		\leq \frac{1}{1-\lambda_2(N)}\log\Big(\frac 4{\pi_{\min}}\Big)
	\end{align*}
	with $\pi_{\min}:=\min_{x\in\T_N}\pi(x)$. In particular, all our results can be read in terms of the spectral gap rather then the mixing time of the chain. 
\end{remark}

\subsection{Notation}\label{sec:not}

We will use the notation $a\wedge b:=\min\{a,\,b\}$ for $a,\,b\in\R$ as well as $a\vee b:=\max\{a,\,b\}$. The use of the symbols $c,\,c_1,\,c_2\ldots$ refers to positive constants whose value may change from line to line. Their value might depend on the model parameters $\alpha$ and $\tau$, but will not depend from other variables (for example, $N$) unless specified otherwise. 

The symbols $\lesssim$ and $\gtrsim$ indicate respectively $\leq $ and $\geq$ eventually up to a slowly varying function in $N$. Actually, except for the upper bound in Theorem \ref{thm:gamma_tra_12}, all the $\lesssim$ and $\gtrsim$ refer to polylogarithmic corrections.

\section{Preliminary results}\label{preliminary}

Let us recall the fundamental inequality
\begin{equation}\label{exp_ineq}
 \exp(x)\ge x+1,\quad x\in \R
\end{equation}
which we will use frequently.

\subsection{The Cheeger constant}\label{sec:cheeger}
We now give a quick overview on the Cheeger constant and its relation to the mixing time. We recall that for the lazy simple random walk on a graph $G=(V,E)$, the bottleneck ratio of a subset $S\subset V$ 
of its state space is defined as~\cite[Remark~7.2]{levinperes}
\begin{align}\label{cmic}
	\Phi(S)=\frac{D_{S,S^c}}{D_S}
\end{align}
	and the Cheeger constant of the chain is defined as
	\begin{align}\label{cheeger}
	\Phi_*=\min_{S:\,\pi(S)\leq 1/2}\Phi(S)
	\end{align}
where $\pi$ is the invariant measure of the walk.
The following result links the mixing time of the chain and its Cheeger constant~\cite[pg. 58]{sinclair2012}:
\begin{align}\label{tmixcheeger}
\frac{1-\Phi_*}{2\Phi_*}\log 4
	\leq \tmix(G) 
	\leq \frac{2}{\Phi_*^2}\log \left(\frac4{\pi_{\min}}\right)
\end{align}
where $\pi_{\min}:=\min_{x\in V} \pi(x)$.

\subsection{Multicommodity flows}\label{paths}
Consider a reversible Markov chain on the vertices $V$ of a graph $G$ with its set of unoriented edges $E$, transition matrix $P$ and reversible measure $\pi$. Let
\[
 \cE(G)=\{(x,\,y)\in V\times V:\,{\{x,y\}\in E}\}
\]
be the set of oriented edges obtained by doubling the unoriented edges.
An $\cE$-path from $x\in V$ to $y\in V$ is a sequence $p=e_1 \,e_2\,\dots\,e_m$  of edges in $\cE(G)$  such that $e_1=(x,\,x_1),\,e_2=(x_1,\,x_2),\,\ldots,\,e_m=(x_{m-1},\,y)$ for some vertices $x_i \in V$. The length of a path $p$ is indicated as $|p|$. The set of all paths is called $\Path$ and the set of all simple paths from $x$ to $y$ is called $\Path(x,\,y)$. We will also use the notation $\Path(G)$ and $\Path(x,\,y,\,G)$ when we need to specify in which graph the path is taken.
A flow is a function $f:\Path\to[0,\,1]$ which satisfies
\[
 \sum_{p\in \Path(x,\,y)}f(p)=\pi(x)\pi(y)\qquad \forall x,\,y\in V,\,x\neq y\,.
\]
Extending the definition of $f$ also to oriented edges, we let the edge load of an edge $e\in \cE(G)$ be
\begin{equation}\label{eq:edge_load}
 f(e):=\sum_{p\in \Path\atop p\ni e}f(p)|p|\,.
\end{equation}
The congestion of a flow $f$  is
\begin{equation}\label{eq:def_congestion}
 \rho(f):=\max_{(a,\,b)\in \mathcal E(G)}\frac{f((a,\,b))}{\pi(a)P(a,\,b)}\,.
\end{equation}
\citet{sinclair1992} establishes the relation between congestion rate and mixing time of the chain (notice that in \citet{sinclair1992} the congestion of a flow $f$ as defined in \eqref{eq:def_congestion} is called $\overline\rho(f)$, while the letter $\rho$ is used for a related  quantity). The author shows 
(see Theorem~5' and Proposition 1 therein) that for any flow $f$ 
\begin{equation}\label{eq:comb_gap}
 \tmix(G)\le \rho(f)\log(4|\cE(G)|)
\end{equation}
and that there exists a flow $f^*$ (see Theorem 8 and Remark (a) therein) such that
\begin{equation}
 {\rho(f^*)}
 \le 32\, \big(\tmix(G)\big)^2\,.\label{eq:sinclair}
\end{equation}

\subsection{A simple lemma} The following lemma will be used for the upper bound of Theorem \ref{thm:gamma_tra_12} to pin the right polynomial order of the mixing time in that regime. It should be a well-known fact about slowly varying functions, but we could not find a reference. 
\begin{lemma}\label{okaramata}
	Let $(X_N)_{N\in\N}$ be a sequence of random variables such that, for each $\varepsilon>0$, 
	\begin{align*}
	\P(X_N>N^{\varepsilon})\xrightarrow{N\to\infty}0\,.
	\end{align*}
Then there exists a slowly varying function $\ell:\,\N\to(0,\infty)$ such that
\begin{align*}
\P(X_N>\ell(N))\xrightarrow{N\to\infty}0\,.
\end{align*}	
\end{lemma}
\begin{proof}
	Take a sequence $(\varepsilon_i)_{i\in\N}$ such that $\varepsilon_i\downarrow 0$. We know that for all $i\in \N$ there exists a $N(i)\in\N$ such that
	\begin{align*}
	\P(X_N>N^{\varepsilon_i})<\varepsilon_i\qquad \forall N\geq N(i)\,.
	\end{align*}
Calling now $\varepsilon(N):=\{\varepsilon_i:\,N(i)\leq N< N(i+1)\}$, we obtain, for $N$ large enough,
	\begin{align*}
\P(X_N>N^{\varepsilon(N)})<\varepsilon(N)\,.
\end{align*}
Since $\lim_{N\to\infty}\varepsilon(N)=0$, we are done if we find a slowly varying function $\ell$ such that 
$$
N^{\varepsilon(N)}\leq \ell(N)\,.
$$ 
By Karamata's representation theorem (see e.g.~\citet[Theorem~1.3.1]{BGTT89}), the function $\ell(N)=\exp\{\sum_{k=1}^N\theta(k)/k\}$ is a slowly varying function as long as $\lim_{N\to\infty}\theta(N)=0$. So it is sufficient to find a function $\theta$ with $\lim_{N\to\infty}\theta(N)=0$ for which
\begin{align*}
\varepsilon(N)\log N\leq\sum_{k=1}^N\frac{\theta(k)}{k}\,.
\end{align*}
This is clearly possible by choosing a $\theta$ which decays to $0$ slowly enough.
\end{proof}

\subsection{Preliminary results on SFP} We will use in different places the following bound on linking probabilities for SFP.
\begin{lemma}\label{bosh}
	For $x,y\in\T_N$ with $\|x-y\|>1$, it holds
	\begin{align}\label{gpand}
	\P(x\leftrightarrow y)\leq\begin{cases}
	c\, \|x-y\|^{-\alpha}\qquad&\mbox{if }\,\tau>2\\
	c\, \|x-y\|^{-\gamma}(\log \|x-y\|)^2 &\mbox{if }\,\tau\leq 2
	\end{cases}
	\end{align}
	where $c>0$ is a constant that only depends on $\alpha$ and $\tau$.
\end{lemma}
\begin{proof}
	Abbreviate $d=\|x-y\|$ and calculate, using \eqref{exp_ineq},
	\begin{align}\label{avera}
	\P(x\leftrightarrow y)
	= \E\Big[1-{\rm e}^{-W_xW_y d^{-\alpha}}\Big]
	\leq \P(W_xW_y>d^\alpha) +\E\Big[W_xW_y d^{-\alpha}\mathbbm{1}_{W_xW_y\leq d^{\alpha}}\Big]\,.
	\end{align}
	The first term on the right hand side of \eqref{avera} is equal to
	\begin{align}\label{luigi}
	\P(W_xW_y>d^\alpha) &=\P(W_x>d^\alpha)+\E\Big[\P\big(W_y>d^{\alpha}W_x^{-1}|W_x\big)\mathbbm{1}_{W_x\leq d^\alpha}\Big] \nonumber\\
	&=d^{-\gamma}+c_1 d^{-\gamma}\int_1^{d^\alpha}w^{-1}\,{\rm d}w
	=d^{-\gamma}(1+c_2\log d)\,.
	\end{align}
	For the second term in \eqref{avera} we must now distinguish between the different regimes. For $\tau>2$ the $W$'s have finite expectation, so 
	\begin{align}\label{shkhn}
	\E\Big[W_xW_y d^{-\alpha}\mathbbm{1}_{W_xW_y\leq d^{\alpha}}\Big]
	\leq c\,d^{-\alpha}
	\end{align}
	for some $c>0$ only depending on $\tau$. For $\tau\leq 2$
	we have
	\begin{align*}
	\E\Big[W_xW_y d^{-\alpha}\mathbbm{1}_{W_xW_y\leq d^{\alpha}}\Big]
	&=	d^{-\alpha}\E\Big[W_x\E\Big[W_y \mathbbm{1}_{W_y\leq d^{\alpha}W_x^{-1}}|W_x\Big]\mathbbm{1}_{W_x\leq d^{\alpha}}\Big]\\
	&=c d^{-\alpha}\E\Big[W_x\mathbbm 1_{W_x\leq d^{\alpha}}\int_1^{d^{\alpha}W_x^{-1}}w^{1-\tau}\,{\rm d}w \Big]\,.
	\end{align*}
	In particular, when $\tau=2$
	\begin{align}\label{bghra}
	d^{-\alpha}\E\Big[W_x\log(d^\alpha W_x^{-1})\mathbbm 1_{W_x\leq d^{\alpha}} \Big]
	\leq d^{-\alpha} \log d^\alpha \int_1^{d^\alpha} w^{-1}\,{\rm d}w
	=d^{-\alpha}(\log d^\alpha)^2\,,
	\end{align}
	while for $\tau<2$ 
	\begin{align}\label{bal}
	d^{-\alpha}\E\Big[W_x\mathbbm 1_{W_x\leq d^{\alpha}}(d^{\alpha}W_x^{-1})^{2-\tau}\Big]
	&=	d^{-\alpha}\E\Big[W_x\mathbbm 1_{W_x\leq d^{\alpha}}(d^{\alpha}W_x^{-1})^{2-\tau}\Big]
	=c \,d^{-\gamma}\log d^\alpha\,.
	\end{align}
	Since $\gamma>\alpha$ for $\tau>2$, $\gamma=\alpha$ for $\tau=2$ and $\gamma<\alpha$ for $\tau<2$, putting  \eqref{shkhn}, \eqref{bghra}, \eqref{bal} together with \eqref{luigi}  into \eqref{avera} yields the desired result.
\end{proof}

\section{Case {$\gamma<1$}: upper bound in Theorem \ref{gammino}}\label{sec:gammino}
In this Section we prove the upper bound of Theorem \ref{gammino}. First of all we will show that the result holds for a simplified model where the geometry of the torus plays no role.  Then we will prove that we can dominate the mixing time of the original model by the square of the mixing time of the simplified model up to a polylogarithmic factor.

\subsection{Proof of Proposition \ref{gammino}}\label{simplifiedmodel}
Under the measure $\P$, we construct another random graph called $\sG_N$ with vertices on the torus $\T_N$. We use  the same random weights $\{W_x\}_{x\in \T_N}$ that we use to construct $G_N$, so the two random graphs are coupled. Then, we put an edge between two vertices in $\sG_N$ if and only if the product of their weights exceeds $N^\alpha(\log N)^2$, that is, for all $x\neq y\in \T_N$,
\begin{align*}
x\leftrightarrow y \quad \Leftrightarrow \quad W_xW_y\geq N^\alpha (\log N)^2\,.
\end{align*}
We refer to $\sG_N$ as the {\it simplified model}. We indicate with $\sD_x$ the degree of a node in $\sG_N$, $\sD_A$ the sum of the degrees of the vertices in $A\subseteq \T_N$, $\spi$ the invariant measure of the lazy simple random walk on $\sG_N$, that is $\spi(x)=\sD_x /\sD_{\sG_N}$. 
Proposition \ref{gmnsmplif} shows that the lazy simple random walk on $\sG_N$ mixes fast, while Proposition \ref{Ub_con_ridotto} tells us that the mixing time on $G_N$ is bounded by the square of that of $\sG_N$ up to a correction. We will prove the two propositions in the next sections and we point out that their combination  yields Theorem \ref{gammino}.

\begin{proposition}\label{gmnsmplif}
	Let $\gamma:=\alpha(\tau-1)<1$. There exists $c>0$ such that $\P$--a.s., for $N$ large enough,
\begin{align*}
\tmix(\sG_N)
	\leq(\log N)^c\,.
\end{align*}
\end{proposition}
\begin{proposition}\label{Ub_con_ridotto}
	$\P$--a.s., for all $N$ large enough, it holds
	\begin{align*}
	\tmix(G_N)\lesssim \big(\tmix(\sG_N)\big)^2\,.
	\end{align*}
\end{proposition} 
\subsection{Preliminary results on the simplified model}

Before proving Propositions \ref{gmnsmplif} and \ref{Ub_con_ridotto}, we collect hereby some facts about $G_N$ and $\sG_N$.

\begin{proposition}\label{coscc}
	$\P$-a.s., for $N$ large enough, the following properties hold:
	\begin{enumerate}[label=(\roman*)]
		\item\label{coscc:i} $\sG_N$ is fully connected.
		\item\label{coscc:ii} $E(\sG_N)\subseteq E(G_N)$.
		
		\item\label{coscc:iii}
		
		${}$		\vspace{-\baselineskip}
		\begin{align}
		\E[D_1\,|\,W_1]
		&\lesssim N^{1-\gamma} W_1^{\tau-1}\;\wedge N,
		\label{expecteddegreeoriginal}
		\\
		 \E[\sD_1\,|\,W_1]
		&=(N-1)\cdot\big\{N^{-\gamma} W_1^{\tau-1}(\log N)^{-2(\tau-1)}\;\wedge \; 1\big\}\label{expecteddegreesimplified}\,.
		\end{align}
		\item\label{coscc:iv}
		For all $x\in\T_N$
		\begin{align}
				\big | D_x - \E[D_x\,|\,W_x]\big| 
		&\lesssim { \E[D_x\,|\,W_x]^{1/2}},\label{degreeconcentrationoriginal}\\
		\big | \sD_x - \E[\sD_x\,|\,W_x]\big| 
			&\lesssim { \E[\sD_x\,|\,W_x]^{1/2}}\,.\label{degreeconcentrationsimplified}
		\end{align}
		\item\label{coscc:v}
		
		${}$		\vspace{-\baselineskip}	
		\begin{align}
		&\big|\sD_{\sG_N}-\E[\sD_{\sG_N}]\big|
			\lesssim N^{(3-\gamma)/2}
			 \label{zmb}\\
			&N^{2-\gamma}\lesssim \E[\sD_{\sG_N}]\lesssim N^{2-\gamma}\,.\qquad
		\label{ashp}
\end{align}
	\end{enumerate}
\end{proposition}
\begin{proof}
	The proof is postponed to Section \ref{appe} in the appendix.
\end{proof}

\subsection{Proof of Proposition \ref{gmnsmplif}}
In view of the upper bound in \eqref{tmixcheeger} and the fact that, for $N$ large enough, $\spi_{\min}\geq 1/N^2$ (using item \ref{coscc:v} in Proposition \ref{coscc}), it will be enough to show that there exists a $c>0$ such that, $\P$--a.s., for $N$ large enough, $\Phi_*(\sG_N)\geq (\log N)^{-c}$ where $\Phi_*(\sG_N)$ is the Cheeger constant (cfr.~\eqref{cheeger}) associated to the lazy simple random walk on $\sG_N$. 
Therefore, we will be done if we can prove the following:
there exists $c>0$ such that, $\P$--a.s.~for $N$ large enough, for each $S\subseteq \T_N$ with $\pi(S)<1/2$ we have 
\begin{align}\label{cheeger_lower}
\Phi(S)	=\frac{\sD_{S,S^c}}{\sD_S}
	\geq (\log N)^{-c}\,,
\end{align}
where in this Section $\Phi(S)$ indicates the bottleneck ratio of the set $S$ associated to the lazy simple random walk on $\sG_N$.

\begin{observation}	\label{osservazione}
	If a set $ S$ is such that $\sD_{S,S^c}\gtrsim N^{2-\gamma}$, we automatically have~\eqref{cheeger_lower}. In fact, $\P$--a.s.~for all $N$ large enough, we have that $\sD_S\leq \sD_{\sG_N}\lesssim N^{2-\gamma}$ by Proposition \ref{coscc} item \ref{coscc:v}.
\end{observation}
\smallskip
We partition  the set of vertices into ``weight slices''. Let 
$$
\jmax:=2+\frac \alpha 2 \frac{\log N}{\log\log N}
\qquad\mbox{ and } \qquad\delta:=2^{-\frac{1}{\tau-1}}
$$ 
and define
\begin{align*}
V_j&:=\{x\in\T_N: \,W_x\in[N^{\alpha/2}(\log N)^j,N^{\alpha/2}(\log N)^{j+1})\}\quad &j=1,\dots,\jmax-1\\
V_{\jmax}&:=\{x\in\T_N: \,W_x\in[N^{\alpha}(\log N)^2,\infty)\}\\
V_{j^c}&:=\{x\in\T_N: \,W_x\in[N^{\alpha/2}(\log N)^{2-j},N^{\alpha/2}(\log N)^{3-j})\}\quad &j=1,\dots,\jmax\\
V_j^+&:=\{x\in\T_N: \,W_x\in[\delta N^{\alpha/2}(\log N)^{j+1},N^{\alpha/2}(\log N)^{j+1})\}\quad &j=1,\dots,\jmax-1\\
V_{j^c}^+&:=\{x\in\T_N: \,W_x\in[\delta N^{\alpha/2}(\log N)^{3-j},N^{\alpha/2}(\log N)^{3-j})\}\quad &j=1,\dots,\jmax
\end{align*}
The $V_j$'s partition the vertices of the graph with weight larger than $N^{\alpha/2}\log N$, whereas the $V_{j^c}$'s partition the vertices with weight smaller than $N^{\alpha/2}(\log N)^2$ (notice that $V_1=V_{1^c}$, while all the other sets $V_j$ and $V_{j^c}$ are mutually disjoint).  $V_j^+$ and $V_{j^c}^+$ are subsets, respectively, of $V_j$ and $V_{j^c}$ with vertices that have a high weight (within their weight slice).
\begin{observation}\label{empln}
In the simplified model, for each $j$, all vertices in $V_{j^c}$ are connected to all vertices in $V_{\ell}$, with $\ell\geq j$, since the product of the weight of a vertex in $V_{j^c}$ and the weight of a vertex in $V_{\ell}$ is always larger than $N^{\alpha}(\log N)^2$.
\end{observation}

\begin{proposition}\label{prop:VjConcentra}
	Call $Q:=(\log N)^{\tau-1}$. $\P$--a.s.~for all $N$ large enough, the following holds:
\begin{enumerate}[label=(\roman*)]
	\item\label{trky}  for  $j=1,\dots,\jmax$
	\begin{align}
	\tfrac 12 N^{1-\gamma/2}Q^{-j}
		&\leq |V_j|
		\leq 2 N^{1-\gamma/2}Q^{-j}\label{urcurc}
		\\
	\tfrac 12 N^{1-\gamma/2}Q^{j-2}
		&\leq|V_{j^c}|
		\leq 2 N^{1-\gamma/2}Q^{j-2}\label{trlero}
		\\
		2 N^{1-\gamma/2}Q^{j-3}
		&\leq |V_{j^c}^+|  
		\leq 2 N^{1-\gamma/2}Q^{j-3}\nonumber
	\end{align}
	and for $j=1,\dots,\jmax-1$
	\begin{align}\label{splsol}
	\tfrac 12 N^{1-\gamma/2}Q^{-(j+1)}
	&\leq|V_{j}^+| 
	\leq 2 N^{1-\gamma/2}Q^{-(j+1)}\,;
	\end{align}
	\item\label{via} {for $j=1,\dots, \jmax-1$},
	\begin{align}\label{D_large_bounds}
	\sD_{V_j}>2\sD_{V_j^+}\qquad \mbox{and}\qquad \sD_{V_{j^c}}>2\sD_{V_{j^c}^+}\,.
	\end{align}
\end{enumerate}	
\end{proposition}
\begin{proof}
	The proof is postponed to Appendix \ref{dib}.
\end{proof}

\medskip

Take now any $S\subseteq \T_N$ and call $Q:=(\log N)^{\tau-1}$ as in Proposition \ref{prop:VjConcentra}. 
We have three possibilities:
\begin{enumerate}
	\item[\bf Case A)] 
		$\quad|V_1\cap S|\in\Big(Q^{-2}|V_1|\,,\,\big(1-Q^{-2})|V_1|\Big)$;
	\item[\bf Case B)]$\quad|V_1\cap S|\geq\big(1-Q^{-2})|V_1|$;
	\item[\bf Case C)] $\quad|V_1\cap S|\leq Q^{-2}|V_1|$.
\end{enumerate}
We analyze separately these three cases and show that \eqref{cheeger_lower} always  holds when $\pi(S)<1/2$. This will conclude the proof.

\medskip

{\bf Case A)} Since all points in $V_1$ are connected to all the other points in $V_1$, in this case we have
\begin{align*}
\sD_{S,\,S^c}
	\ge \sD_{V_1\cap S,\,V_1\cap S^c}
	\ge |V_1\cap S|\,|V_1\cap S^c|
	\gtrsim |V_1|^2
	\stackrel{\eqref{urcurc}}{\gtrsim} N^{2-\gamma}\,,
\end{align*}
so that we have  \eqref{cheeger_lower} thanks to Observation \ref{osservazione}.

\medskip

{\bf Case B)} Define 
$$
\jbar:=\inf\big\{j:\,|V_j\cap S|<\big(1-Q^{-2})|V_j|\big\} \,.
$$
We distinguish two subcases: either $\jbar\in\{2,\dots,\jmax\}$ or $\jbar=\infty$ (that is, the set on the right-hand side above~is empty).

\smallskip

\noindent {\bf Case B1)} Suppose that $\jbar\in\{2,\dots,\jmax\}$. Recall (cfr.~Observation \ref{empln}) that, by construction, all the points in $V_{\jbar-1}$ and all the points in $V_{\jbar}$ are connected to all the points in $V_{(\jbar-1)^c}$. There are now two possibilities:
if 
$|V_{(\jbar-1)^c}\cap S|  \leq  \big(1-Q^{-2})|V_{(\jbar-1)^c}|$, that is, most of the points of $V_{(\jbar-1)^c}$ are in $S^c$, then we can bound
\begin{align*}
\sD_{S,S^c}
	\geq|V_{\jbar-1}\cap S|\, |V_{(\jbar-1)^c}\cap S^c| 
	\geq (1-Q^{-2}) |V_{\jbar-1}|\cdot Q^{-2}|V_{(\jbar-1)^c}|\,,
\end{align*}
otherwise we bound
\begin{align*}
\sD_{S,S^c}
\geq |V_{(\jbar-1)^c}\cap S| \,|V_{\jbar}\cap S^c|\, 
\geq (1-Q^{-2}) |V_{(\jbar-1)^c}|\cdot Q^{-2}|V_{\jbar}|\,.
\end{align*}
Using \eqref{urcurc} and \eqref{trlero}, we see that in both cases \eqref{cheeger_lower} holds thanks to Observation \ref{osservazione}.

\smallskip

\noindent {\bf Case B2)} Suppose now that $\jbar=\infty$. If there exists $i\in\{1,\dots,\jmax\}$ such that $|V_{i^c}\cap S|  \leq  \big(1-Q^{-2})|V_{i^c}|$, then as before
\begin{align*}
\sD_{S,S^c}
\geq |V_{i}\cap S| \,|V_{i^c}\cap S^c|\, 
\geq (1-Q^{-2}) |V_{i}|\cdot Q^{-2}|V_{i^c}|
\end{align*}
and we obtain again \eqref{cheeger_lower} with \eqref{urcurc} and \eqref{trlero}. If such an $i$ does not exist, it means that $|U\cap S|\geq (1-Q^{-2}) |U|$ for all $U\in\{V_1,V_{1^c},\dots,V_{\jmax},V_{\jmax^c}\}$. We want to show now that this kind of set $S$ is not to be taken into account in \eqref{cheeger} since $\pi(S)>1/2$. This should intuitively be true because we are considering a set $S$ that contains the large majority of the points of each weight slice. More precisely we decompose (recall that $D_{V_1}=D_{V_{1^c}}$)
\begin{align*}
\overline\pi(S)
=\frac{\sD_{V_1\cap S}+\sum_{j=2,\dots,\jmax} \big(\sD_{V_j\cap S}+\sD_{V_{j^c\cap S}}\big)}{\sD_{V_1}+\sum_{j=2,\dots,\jmax} \big(\sD_{V_j}+\sD_{V_{j^c}}\big)}\,.
\end{align*}
Consequently, if we can prove that 
\begin{align}\label{wkknd}
\sD_{V_j\cap S}> \sD_{V_j}/2\quad \mbox{and} \quad \sD_{V_{j^c}\cap S}> \sD_{V_{j^c}}/2\qquad \forall j=1,\dots,\jmax 
\end{align}
it will follow that $\overline\pi(S)>1/2$. We will show \eqref{wkknd} just for the $D_{V_j}$'s, the proof for the $D_{V_{j^c}}$ being completely similar.

For $j=\jmax$, we notice that  most of the points of $V_{\jmax}$ are inside the set $S$. Since all of the points in $V_{\jmax}$ have the same degree (which is $N-1$), we clearly have $\sD_{V_{\jmax}\cap S}> \sD_{V_{\jmax}}/2$.
Take now $j\in\{1,\dots,\jmax-1\}$.
By \eqref{urcurc} and \eqref{splsol} we know that 
 \begin{align}\label{plna}
|V_{j}^+|
	\geq \tfrac{1}{4}Q^{-1}|V_j|
 	\geq Q^{-2}|V_j|\,.
 \end{align}
Since by assumption $|V_j\cap S^c|\leq Q^{-2}|V_j|$ 
and since the vertices in $V_j^+$ are the vertices of $V_j$ with the largest weight and therefore with the largest degree, we deduce that $\sD_{V_j\cap S}\geq \sD_{V_j\setminus V_j^+}$. But $\sD_{V_j\setminus V_j^+}\geq \sD_{V_j}/2$ by \eqref{D_large_bounds}, so that all in all $\sD_{V_j\cap S}> \sD_{V_j}/2$.

\medskip

{\bf Case C)} We mirror the proof of Case B. Define
\begin{align*}
\jbarbar:=\inf\big\{ j\,:\,|V_j\cap S|\geq Q^{-2}|V_j| \big\}\,.
\end{align*}
As for $\jbar$, $\jbarbar$ can be either finite or not.

\smallskip

\noindent {\bf Case C1)}  Suppose $\jbarbar\in\{2,\dots,\jmax\}$. Then we can reason similarly to Case B1: 

If $|V_{(\jbarbar-1)^c}\cap S |\geq Q^{-2}|V_{(\jbarbar-1)^c}|$, then 
\begin{align*}
\sD_{S,S^c}
	\geq |V_{(\jbarbar-1)^c}\cap S| \,|V_{\jbarbar-1}\cap S^c|
	\geq Q^{-2}|V_{(\jbarbar-1)^c}|\cdot(1-Q^{-2}) |V_{\jbarbar-1}|\,,
\end{align*}
otherwise we bound
\begin{align*}
\sD_{S,S^c}
	\geq |V_{\jbarbar}\cap S|\,|V_{(\jbarbar-1)^c}\cap S^c| 
	\geq Q^{-2}|V_{\jbarbar}|\cdot(1-Q^{-2}) |V_{(\jbarbar-1)^c}|\,.
\end{align*}
In both cases, using \eqref{urcurc} and \eqref{trlero}, we arrive to \eqref{cheeger_lower} by Observation \ref{osservazione}\,.

\smallskip

\noindent {\bf Case C2)} Finally suppose that $\jbarbar=\infty$. Similarly to Case B2, if there exists $i\in\{1,\dots,\jmax\}$ such that $|V_{i^c}\cap S|  \geq  Q^{-2}|V_{i^c}|$, then 
\begin{align*}
\sD_{S,S^c}
	\geq |V_{i^c}\cap S|\,|V_{i}\cap S^c| \, 
	\geq Q^{-2}|V_{i^c}|\cdot(1-Q^{-2}) |V_{i}|
\end{align*}
and we obtain once more \eqref{cheeger_lower}. 
If such an $i$ does not exist, it means that $|U\cap S|\leq Q^{-2} |U|$ for all $U\in\{V_1,V_{1^c},\dots,V_{\jmax},V_{\jmax^c}\}$. 
Then, setting $V_{(\jmax+1)^c}=\emptyset$ to ease the notation, 
on the one hand
\begin{align}\label{brnld}
\sD_{S,S^c}
	&\geq \sum_{j=1}^{\jmax}|V_j\cap S|\Big(\sum_{i\leq j}|V_{i^c}\cap S^c|+\sum_{i=1}^{\jmax}|V_i\cap S^c|\Big)\nonumber\\
	&\geq \big(1-Q^{-2}\big)\sum_{j=1}^{\jmax}|V_j\cap S|\Big(\sum_{i\leq j}|V_{i^c}|+\sum_{i=1}^{\jmax}|V_i|\Big)
\end{align}
while on the other hand
\begin{align}\label{rmngg}
\sD_{S,S}
	&\leq \sum_{j=1}^{\jmax}|V_j\cap S|\Big(\sum_{i\leq j+1}|V_{i^c}\cap S|+\sum_{i=1}^{\jmax}|V_i\cap S|\Big)\nonumber \\
	&\leq Q^{-2}\sum_{j=1}^{\jmax}|V_j\cap S|\Big(\sum_{i\leq j}|V_{i^c}|+|V_{{(j+1)}^c}|+\sum_{i=1}^{\jmax}|V_i|\Big)\,.
\end{align}
Since by \eqref{trlero} one has
$|V_{{(j+1)}^c}|\leq 4Q|V_{j^c}|$, one sees from \eqref{brnld} and \eqref{rmngg} that
\begin{align*}
\sD_S
	\leq \frac{Q^{-2}(4Q+1)}{1-Q^{-2}}\sD_{S,S^c}
	\leq Q^{-1/2}\sD_{S,S^c}\,.
\end{align*}
Therefore
\begin{align*}
\Phi(S)
	=\frac{ \sD_{S,S^c}}{ \sD_S}
	=\frac{ \sD_{S,S^c}}{ \sD_{S,S^c}+\sD_{S,S}}
	\geq c
\end{align*}
for some constant $c>0$. This concludes the proof of Proposition \ref{gmnsmplif}.

\subsection{Proof of Proposition \ref{Ub_con_ridotto}} 
For a realization of the simplified model $\sG_N$, let {$\overline f^*$ be a flow on $\sG_N$ for which \eqref{eq:sinclair} holds}.
We define now a flow on $G_N$ as follows: for each $x,y\in\T_N$ and each path $p$ from $x$ to $y$ we set 
\begin{align*}
f(p)=
\begin{cases}
\overline f^*(p)\frac{\pi(x)\pi(y)}{\spi(x)\spi(y)}\qquad&\mbox{if $p$ is allowed in $\sG_N$}\\
0\qquad&\mbox{otherwise.}
\end{cases}
\end{align*}
Recall by Proposition~\ref{coscc:ii} that $E(\sG_N)\subseteq E(G_N)$ almost surely for $N$ large enough. Since $\sG_N$ is connected for $N$ large enough, it is easy to check that $f$ is indeed a flow for $G_N$. 

We notice that, by the definition of $f$, for any edge $e\in E(G_N)\cap E(\sG_N)$, 
\begin{align*}
f(e)
	\leq \Big(\max_{z\in\T_N}\frac{\pi(z)}{\spi(z)}\Big)^2 \overline f^*(e)
\end{align*}
while for any edge $e\in G_N\setminus \sG_N$ one has $f(e)=0$.
Using also that $\sD_{\sG_N}\leq D_{G_N}$
, we  can bound the congestion of $f$ by
\begin{align*}
\rho(f)
	\leq\frac{D_{G_N}}{\sD_{\sG_N}}\Big(\max_{z\in\T_N}\frac{\pi(z)}{\spi(z)}\Big)^2\rho(\overline f^*)
	{\stackrel{\eqref{eq:sinclair}}{\le} \Big(\max_{z\in\T_N}\frac{D_z}{\sD_z}\Big)^2\big(\tmix(\overline G_N)\big)^2}\,.
\end{align*}
By~\eqref{eq:comb_gap} and the fact that $\sD_{\sG_N}\lesssim N^{2-\gamma}$ (as can be seen by \eqref{zmb} and \eqref{ashp}), we are done if we prove that for some $c>0$
\begin{align}\label{corri}
\max_{z\in\T_N}\frac{D_z}{\sD_z}\leq (\log N)^c
\end{align}
almost surely for $N$ large enough.
But this is true, since each node $z$ has the same weight $W_z$ in both models, so \eqref{expecteddegreeoriginal}, \eqref{expecteddegreesimplified}, \eqref{degreeconcentrationoriginal} and \eqref{degreeconcentrationsimplified} do the job.

\section{Case {$1<\tau<2,\,1<\gamma<2$}: Upper bound of Theorem~\ref{thm:gamma_tra_12}. First part.}\label{sec:importante}

We fix some
\begin{align}\label{miktti}
0<\varepsilon < (\gamma-1) \wedge \frac{2-\gamma}{2}
\end{align}
and let  
$$
L=L(N,\gamma,\varepsilon):=\lfloor N^{\gamma-1+\varepsilon}\rfloor
\qquad\mbox{and}\qquad
K=K(N,\gamma,\varepsilon)=\tfrac{N-\ell}{L}
$$
with $\ell= N \mod L$. We divide $\T_N$ into $K$  ``chunks'' $S_1,\dots,S_K$, with 
$$
S_i:=\{(i-1)L,\,\dots,\,iL-1\}
$$ 
for all $i=1,\,\dots,\,K-1$ and $S_K:=\{(K-1)L,\,\dots,\,N\}$. Notice that all chunks have length $L$ except for the last one with length $L+\ell< 2L$. 

\begin{remark}\label{rem:chunk_l}
	For simplicity of exposition we will consider from now on $\ell=0$, so that all the chunks have the same size. Our results will still hold, up to minor corrections, for other values of $\ell$. When the needed modifications will not be minor, we will explicitly indicate what has to be changed. 
\end{remark}
We call $\Gamma=\Gamma(N)$ the random graph on $\T_K$ obtained in the following way: for a realization of SFP $G_N$, 
 we put a single undirected edge between node $i$ and $j$ in $\Gamma$ 
if the point $x_{\max}(i)$ with the largest weight in $S_i$ is connected to the point $x_{\max}(j)$ with the largest weight in $S_j$. That is, 
\begin{align}\label{mipe}
x_{\max}(i)\stackrel{G_N }{\leftrightarrow} x_{\max}(j)
\qquad\Longrightarrow\qquad
	i\stackrel{\Gamma }{\leftrightarrow} j,\quad i\in\T_K.
\end{align}
We include in any case all edges between neighbouring $i$'s, too.

\begin{lemma}\label{keylemma}
	Let $\tGamma=\tGamma(N)$ be a SFP random graph on $\T_K$ with parameters 
	\begin{align}\label{tildami}
	\tilde \tau=\tau \qquad\mbox{and}\qquad \tilde\alpha=\frac{2-\gamma-\tfrac 32\varepsilon}{(\tau-1)(2-\gamma-\varepsilon)}\,.
	\end{align}
	It is possible to couple $G_N$ and $\tGamma$ so that
	\begin{align}\label{prlmn}
		\P(E(\Gamma)\supseteq E(\tGamma))=1\,.
	\end{align}
	Furthermore, it holds that
	\begin{align}\label{ttgl}
	\tmix(G_N)
	\lesssim \Delta_{G_N}|E(G_N)|\Big(\Pi_{G_N	}
	+R_{G_N,\tGamma}^2\frac{1}{|E(\tGamma)|}\big(\tmix(\tGamma)\big)^2\Big)\,,
\end{align}
	where
	\begin{align*}
		\Delta_{G_N}:= \max_{j\in\T_K}{\rm diam}(G^j_N)\qquad
		\Pi_{G_N}:=	\max_{j\in\T_K}\sum_{z\neq w\in S_j}\pi_{G_N}(z)\pi_{G_N}(w)\qquad
		R_{G_N,\tGamma}:= \max_{j\in\T_K}\frac{\pi_{G_N}(S_j)}{\pi_{\tGamma}(j)}
	\end{align*}
	with $G^j_N$ indicating the restriction of $G_N$ on $S_j$.
\end{lemma}

Notice that $\tGamma$ in the lemma is an SFP where $
\tgamma=\tilde\alpha(\tilde\tau-1)<1\,.
$
Therefore we already know that $\tmix(\tGamma)$ is polynomial in $K$ (and hence in $N$) by Theorem \ref{gammino}. We point out that the presence of $\varepsilon$ in the definition of $L$ is needed for the stochastic domination~\eqref{prlmn}.

\smallskip

We proceed now as follows. We will spend the rest of this section in the proof of Lemma \ref{keylemma}, dividing the proof in two parts: in Subsection \ref{stochasticdomination} we describe how to couple the two random graphs and obtain \eqref{prlmn}, while in Subsection \ref{nonstochasticdomination} we derive \eqref{ttgl}.
Afterwards, in Section \ref{sec:importante2}, we will show how the upper bound in Theorem \ref{thm:gamma_tra_12} follows from Lemma \ref{keylemma}.

\subsection{Stochastic domination, proof of {\eqref{prlmn}}}\label{stochasticdomination}
We show now that we can build $\tGamma$ on the same probability space of $G_N$ (and therefore of $\Gamma$) so that condition \eqref{prlmn} is satisfied. The idea is to couple the weights $\tW_j$ in $\tGamma$ with the weights in $\Gamma$ so that, roughly, $\tW_j\sim L^{-{1}/(\tau-1)} \Wmax(S_j)$, where $\Wmax(S_j):=\max_{x\in S_j}W_x=W_{x_{\max}(j)}$. Since we will use this fact in the next section, we rephrase the statement in a more precise proposition.

 \begin{proposition}\label{prop:stochasticdomination}
	There exists a coupling between $G_N$ and $\tGamma$ such that, $\P$--a.s.~for all $N$ large enough,
		\begin{align}\label{pfa}
	\Big({L^{-\tfrac{1}{\tau-1}}}{\Wmax (S_j)}\,\vee \,1\Big)
	\leq \,\tW_j\,
	\leq (\log L)^{\tfrac 2{\tau-1}}{L^{-\tfrac{1}{\tau-1}}}{\Wmax (S_j)}
	\end{align}
	and 
	\begin{align}\label{pla}
	\P\big(E(\Gamma)\supseteq E(\tGamma)\big)=1\,.
	\end{align}
\end{proposition}
\begin{proof}
	First of all we check that it is possible to couple the weights of the nodes so that \eqref{pfa} holds. To this end, it is enough to show that
	\begin{align}
	\P(\tW_j>t)
		&\geq \P\big({L^{-{1}/({\tau-1})}}{\Wmax (S_j)}>t\big) &\forall t\geq 1\label{pino}\\
	 \P(\tW_j>t)
		&\leq \P\big(L^{-{1}/({\tau-1})}U_L\Wmax (S_j)>t\big) &\forall t\geq 1\label{kro}
	\end{align}
	where we shortened $U_L:=(\log L)^{ 2/{\tau-1}}$.
	For \eqref{pino} we notice that, for all $t\geq 1$, the right hand side is 
	\begin{align*}
		1-\P\big(W_x\leq tL^{{1}/({\tau-1})}\big)^L
		=1-\big(1-t^{-(\tau-1)}L^{-1}\big)^L
		\leq t^{-(\tau-1)}
		=\P(\tW_j>t)
	\end{align*}
	where for the inequality we used the fact that $(1-a)^L\geq 1-aL$ for all $a\geq 0$ and $L\geq 0$. 
	For  \eqref{kro} we calculate
	\begin{align*}
	\P\big(L^{-{1}/({\tau-1})}U_L\Wmax (S_j)>t\big)
		&=1-\big(1-(t/U_L)^{-(\tau-1)}L^{-1}\big)^L
			\geq 1-\exp\{-(\log L)^2t^{-(\tau-1)}\} \,.
	\end{align*}
	This quantity is larger than $\P(\tW_j>t)=t^{-(\tau-1)}$ for, say, all $t\geq 2$ as can be checked straightforwardly. But we claim that $L^{-{1}/({\tau-1})}U_L\Wmax (S_j)$ is always larger than $2$, for $L$ large enough. In fact,
	\begin{align*}
	\P\big(\exists j\in\T_K:\,L^{-{1}/({\tau-1})}U_L\Wmax (S_j)\leq 2\big)
		&\leq K \big(1-(2/U_L)^{-(\tau-1)}L^{-1}\big)^L 
		\leq  K\e^{-(\log L)^2/2^{\tau-1}}
	\end{align*}
	which is summable in $L$ (recall that $K$ can be written as a polynomial in $L$), so the first Borel-Cantelli lemma gives the claim and \eqref{kro} is verified for $t<2$, too.
	
	\smallskip
	
	Now we consider $G_N$ and $\tGamma$ built on the same probability space with weights satisfying \eqref{pfa}. We recall that, given the weights, the presence of each edge is independent from the others. Hence, showing that  for all $i\neq j\in\T_K$
	\begin{align}\label{elio}
	\P\big(i\stackrel{\Gamma }{\leftrightarrow} j \,|\,\{W_x\}_{x\in \T_N}\big)
		\geq \P\big(i\stackrel{\tGamma }{\leftrightarrow} j \,\big|\,\{\tW_\ell\}_{\ell\in\T_K}\big)
	\end{align}
	will ensure that there exists a coupling such that \eqref{pla} holds.
 	Without loss of generality we can take $i=1$ and show \eqref{elio} for all $j=3,\dots,\lceil K/2\rceil$ (the case $j=2$ is trivial, since by the definition of the model all nearest neighbours are connected with probability $1$, and we stop at $\lceil K/2\rceil$ since we are dealing with the torus distance). Recalling \eqref{mipe}, the left-hand side~of \eqref{elio} can be bounded as
 	\begin{align*}
 	\P(1\stackrel{\Gamma }{\leftrightarrow} j \,|\,\{W_x\}_{x\in \T_N})
 		&\geq 1-\exp\Big\{-  \Wmax(1)\Wmax(j)\,(jL)^{-\alpha} \Big\}\\
 		&\stackrel{\eqref{pfa}}{\geq} 1-\exp\Big\{-  {\tW_1\tW_j}\,{j^{-\alpha} L^{-\alpha+\tfrac{2}{\tau-1}}}(\log L)^{-\tfrac{4}{\tau-1}}\Big\}
 	\end{align*}
	where for the first inequality we have used the fact that for all $x\in S_1$ and $y\in S_j$ one has $\|x-y\|\leq jL$. On the other hand, the r.h.s.~of \eqref{elio} is
	\begin{align*}
	\P\big(1\stackrel{\tGamma }{\leftrightarrow} j \,\big|\,\{\tW_\ell\}_{\ell\in\T_K}\big)
		= 1-\exp\Big\{- \tW_1\tW_j \,(j-1)^{-\talpha}\Big\},
	\end{align*}
	so \eqref{elio} is verified if we prove that
	\begin{align}\label{jnfr}
	{j^{-\alpha} L^{-\alpha+\tfrac{2}{\tau-1}}}(\log L)^{-\tfrac{4}{\tau-1}}
		\geq (j-1)^{-\talpha}\,.
	\end{align}
	Since $\alpha>\talpha$, it is enough to show \eqref{jnfr} for $j=\lceil K/2 \rceil $. Recalling that $L=\lfloor N^{\gamma-1+\varepsilon}\rfloor$ and $K=N/L$, we see that 
	\begin{align}\label{grdn}
	\mbox{l.h.s.~of \eqref{jnfr}}
	\gtrsim
	N^{-\alpha}L^{\tfrac 2{\tau-1}}
		= N^{-\tfrac{2-\gamma-2\varepsilon}{\tau-1}}
	\end{align}
	while, recalling \eqref{tildami},
	\begin{align}\label{brn}
	\mbox{r.h.s.~of \eqref{jnfr}}
		\lesssim K^{-\talpha}
		\lesssim  N^{-\tfrac{2-\gamma- 3 \varepsilon/2}{\tau-1}}\,.
	\end{align}
	Comparing \eqref{grdn} and \eqref{brn}, we obtain that \eqref{jnfr} holds for all $N$ large enough,  which in turn gives \eqref{elio} and so \eqref{pla}.
\end{proof}

\subsection{Comparison of the mixing times, proof of~{\eqref{ttgl}}}\label{nonstochasticdomination}
We show now how to obtain \eqref{ttgl}, closing the proof of Lemma \ref{keylemma}.
The idea comes from \cite{BBY08} (cfr.~Proposition~2.1 thereby), and we will borrow part of its notation; we will also drop the $N$ from $G_N$ and $G_N^i$ for simplicity. Recall Section \ref{paths} for some notation about paths.
For $x,\,y$ both in some  $S_i$ denote as $p(x,\,y)$ the graph--geodesic in $G^i$ between $x$ and $y$ (if there is more than one, we just choose any). 
For an edge $(i,j)\in E(\Gamma)$, let $e(i,j)\in E(G)$ be the edge $(x_{\max}(i),x_{\max}(j))$.
If $q=e_1 e_2\cdots e_{|q|}\in\Path(i,\,j,\,\tGamma)$ and $x\in S_i$, $y\in S_j$, we denote $p(q,\,x,\,y)$ the path in $G$ from $x$ to $y$ that uses $p(x,x_{\max}(i))$, then the edges induced by $q$ and then $p(x_{\max}(j),y)$, that is
\[
p(q,\,x,\,y):=p(x,\,e_1^+)e(e_1^+,\,e_1^-)e(e_1^-,\,e_2^+)\cdots e(e_{|q|}^+,\,e_{|q|}^-)p(e_{|q|}^-,\,y)
\]
where for each oriented edge $e_k$ we indicate with $e_k^+$ (respectively $e_k^-$) its starting (ending) point. We also observe that for any $q,\,x,\,y $ as above one has
\begin{equation}
\label{eq:diam_path}
|p(q,\,x,\,y)|\le 2\Delta_G+|q|\leq 2\Delta_G |q|\,.
\end{equation}
Let
$f^*$ denote 
a flow on $\tGamma$ for which \eqref{eq:sinclair} holds.
From it, we will construct a flow $f$ on $G$ as follows: 
\begin{itemize}[leftmargin=*]
	\item[-] for $x,\,y\in S_i$, set $f(p):=\pi_G(x)\pi_G(y)$ if $p=p(x,\,y)$, and $0$ otherwise;
	\item[-] for $x\in S_i$, $y\in S_j$ with $i\neq j$, set for any $q\in\Path(i,\,j,\,\tGamma)$
	\[
	f(p(q,\,x,\,y))=\frac{f^\ast(q)}{\pi_{\tGamma}(i)\pi_{\tGamma}(j)}\pi_G(x)\pi_G(y) \]                                                                      
	and $0$ otherwise.
\end{itemize}
It is straightforward to verify that this defines a flow on $G$. Let us now compute the congestion rate associated to $f$. Let $(x,y)\in E(G)$ and let $x\in S_i$ and $y\in S_j$.
\begin{itemize}[leftmargin=0.8cm]
	\item If $i\neq j$,  then denoting $q^+$ (respectively $q^-$) the starting (ending) vertex of a path $q$ we obtain
	\begin{align}
	\sum_{p\in\Path(G)\atop p\ni (x,\,y)}f(p)|p|
	&\stackrel{\eqref{eq:diam_path}}{\le} 2\Delta_G\sum_{q\in \Path(\Gamma)\atop q\ni (i,\,j)}\sum_{z\in S_{q^+}\atop w\in S_{q^-}}\frac{f^\ast(q)|q|}{\pi_{\tGamma}(q^+)\pi_{\tGamma}(q^-)}\pi_G(z)\pi_G(w)\nonumber\\
	&\le 2\Delta_G\, R_{G,\,\tGamma}^2 \sum_{q\in \Path(\Gamma)\atop q\ni (i,\,j)}f^\ast(q)|q|\nonumber\\
	&\stackrel{\eqref{eq:sinclair}}{\le}
		c\Delta_G\, R_{G,\,\tGamma}^2 \frac{1}{|E(\tGamma)|} \big(\tmix(\tGamma)\big)^2\label{eq:bound_one_subt}
	\end{align}
	for some constant $c>0$.
	\item If $i=j$, any path $p$ that contains the edge $(x, y)$ such that $f (p) > 0$ must be of the form $p = p(z, \,w)$ for some $z,\, w \in S_i$. Therefore 
	\begin{align}\label{eq:bound_two_subt}
	\sum_{p\in\Path(G)\atop p\ni (x,\,y)}f(p)|p|
	&\le \sum_{z{\neq} w\in S_i}f(p(z,\,w))|p(z,\,w)|
	\leq \Delta_G\sum_{z\neq w\in S_i}\pi_G(z)\pi_G(w) 
	\le \Delta_G \,\Pi_G\,.
	\end{align}
\end{itemize}
Now note that for a flow $f$ one has that
\begin{equation}\label{eq:bound_rho_edge}
\rho(f)
	\stackrel{
		\eqref{eq:def_congestion}}{\le}  
	4 |E(G)| \max_{(a,\,b)\in E}\sum_{p\in\Path(G)\atop p\ni (a,\,b)}f(p)|p|\,.
\end{equation}
The result~\eqref{ttgl} follows by applying~\eqref{eq:comb_gap} and~\eqref{eq:bound_rho_edge} to $\tmix(G_N)$ and then using~\eqref{eq:bound_one_subt} and \eqref{eq:bound_two_subt}.

\section{Case {$1<\tau<2,\,1<\gamma<2$}: Upper bound of Theorem~{\ref{thm:gamma_tra_12}}. Second part.}\label{sec:importante2}

In the previous section we showed inequality \eqref{ttgl}. In order to conclude the  proof of the upper bound of Theorem~\ref{thm:gamma_tra_12}, we will have therefore to bound the quantities appearing on the right hand side of \eqref{ttgl}.
As mentioned before
\begin{align}\label{sbrot}
\tmix(\tGamma)\leq (\log N)^{c}
\end{align}
since, for $\varepsilon$ small enough, $\tGamma$ is a SFP random graph with $\tgamma=\talpha(\ttau-1)<1$: by using the result of Theorem \ref{gammino} its mixing time is at most logarithmic in the number of nodes, which is $K$.
We are left to control the quantities $\Delta_{G_N}$, $\Pi_{G_N}$, $R_{G_N,\tGamma}$, ${|E(G_N)|}$ and ${|E(\tGamma)|}$. This is taken care of in the next four propositions. Also in this section we will drop the $N$ from $G_N$ to ease the notation.

\begin{proposition}\label{diametro}
	Let $\gamma<2$. Recall that $\Delta_{G_N}:= \max_{j\in\T_K}{\rm diam}(G^j_N)$.
	There exists $c>0$ such that, $\P$-a.s.~for all $N$ large enough,
	\begin{align}\label{tos}
	\Delta_G\leq (\log N)^c\,.
	\end{align}
\end{proposition}
\begin{observation}
	Mind that, for completeness, Proposition~\ref{diametro} is stated for a set of parameters that is more general than the one considered in this Section, that is, we are not imposing that  $1<\tau<2$. 
\end{observation}

\begin{proposition}\label{ratiodegrees}
	There exists $c>0$ such that, $\P$-a.s.~for all $N$ large enough, 
	\begin{align}\label{pscl}
	N^{\gamma-1}
		\leq  \frac{|E(G)|}{|E(\tGamma)|}
		\leq  N^{\gamma-1+c\varepsilon}.
	\end{align}
\end{proposition}

\begin{proposition}\label{ratiopis}
	Recall that $R_{G,\tGamma}:=
	\max_{j\in\T_K} {\pi_{G}(S_j)}/{\pi_{\tGamma}(j)}$. There exists a constant $c>0$ such that, $\P$-a.s.~for all $N$ large enough,
	\begin{align}\label{gglne}
	R_{G,\tGamma}
		\lesssim  N^{\varepsilon}
	\end{align}
\end{proposition}

\begin{proposition}\label{ags}
	Recall that $\Pi_{G}:=
	\max_{j\in\T_K}  \sum_{z\neq w\in S_j}\pi_{G}(z)\pi_{G}(w)$. 
	It holds
		\begin{align}\label{gglne2}
	\P\big(|E(G)|\Pi_{G}
	>  N^{\gamma-1}\big)
	\xrightarrow{N\to\infty}0\,.
	\end{align}
\end{proposition}

Before giving the proofs of these four propositions in the next subsections, we conclude the argument for the upper bound of Theorem \ref{thm:gamma_tra_12}.
Using \eqref{sbrot}, \eqref{tos}, \eqref{pscl}, \eqref{gglne} and  \eqref{gglne2} in combination with \eqref{ttgl} yields that there exist positive constants $c_1$ and $c_2$ not depending on $N$ such that
\begin{align*}
\P(\tmix(G_N)\geq (\log N)^{c_1}N^{\gamma-1+c_2\varepsilon})
	\xrightarrow{N\to\infty}0\,.
\end{align*}
Since $\varepsilon$ can be taken arbitrarily small, we finally obtain the upper bound of Theorem \ref{thm:gamma_tra_12} thanks to Lemma \ref{okaramata}.

\subsection{Proof of Proposition \ref{diametro}}
	The proof of the following proposition takes inspiration from the argument of~\citet[Theorem~5.1]{deijfen} which deals with graph distances in SFP. The approach is alternative to the renormalization approach of~\citet[Theorem 3.1]{benjaberger}. First of all we consider $G^1$, the graph induced by $G$ on $\{1,\dots,L\}$.
	Take a constant $M>0$, to be chosen large enough later on. We denote $\mathcal D(\cdot,\,\cdot)$ the graph distance between points of $G^1$.
	We start by showing that, for some $\xi>0$ and $c>0$,
	\begin{align}\label{unoelle}
	{\P\big(\mathcal D(1,L)>3(\log L)^M\big)
		\leq \e^{-c(\log L)^{\xi M} }\,.}
	\end{align}
	Fix some $0<\delta<(2-\gamma)/2$ and define 
	$$
	\imax:=\big\lfloor\log_2 \big(\tfrac{L}{(\log L)^M}\big)\big\rfloor\,.
	$$
	 For $i=2,\dots,\imax$  let
	\begin{align*}
	A_i:=[2^{-i-1}L,\,2^{-i}L)\cap \N
	\qquad\mbox{and}\qquad
	B_i:=[L-2^{-i}L,\,L-2^{-i-1}L)\cap\N
	\end{align*}
	and $A_1=B_1=[\tfrac 14 L,\,\tfrac 34 L)\cap \N$. We want to show that in each $A_i$ (respectively in each $B_i$) there is a point $a_i$ (resp.~$b_i$) such that $a_i\leftrightarrow a_{i+1}$ (resp.~$b_i\leftrightarrow b_{i+1}$) with high probability. Notice that if the event $\bigcap_i \big\{a_i\leftrightarrow a_{i+1},\,b_i\leftrightarrow b_{i+1}\big\}$ happens, then $\mathcal D(1,L)\leq 2(\imax + (\log L)^M)\leq 3(\log L)^M$, since $(\log L)^M$ is an upper bound of the distance between $1$ and the rightmost point of $A_{\imax}$ (resp.~between the leftmost point of $B_{\imax}$ and $L$), being neighbouring points always connected. We focus on the $A_i$'s, since for the $B_i$'s the same  calculation holds.
	
	\smallskip
	
	Let $a_i$ be the point in $A_i$ with the largest weight:
	\begin{align*}
	a_i:=\arg \max_{x\in A_i} W_x
	\end{align*}
	and let $F$ be the event
	\begin{align*}
	F:=\{W_{a_i}\geq |A_i|^{\frac{1-\delta}{\tau-1}},\quad\forall \,i=1,\dots,\imax \}
	\end{align*}
	where  $|A_i|$ indicates as usual the cardinality of the set $A_i$.
	We bound
	\begin{align*}
	\P(F^c)
	\leq\sum_{i=1}^{\imax}\P\big(W_{a_i}<|A_i|^{\frac{1-\delta}{\tau-1}}\big)
	\leq\sum_{i=1}^{\imax}\Big(1-|A_i|^{-1+\delta}\Big)^{|A_i|}
	\leq \sum_{i=1}^{\imax}\e^{-|A_i|^{\delta}}\,.
	\end{align*}
	By upper bounding the last expression with $\imax$ times the largest summand (which corresponds to $i=\imax$) and noticing that $|A_{\imax}|\leq  (\log L)^M$ we find
	\begin{align}\label{dritto}
	\P(F^c)\leq \exp\big\{-c(\log L)^{M\delta}\big\}
	\end{align}
	for some $c>0$ not depending on $L$.
	
	On the other hand, conditioning on $F$, it is unlikely that for some $i$ one has $a_i\not\leftrightarrow a_{i+1}$:
	\begin{align}\label{ingrm}
	\P\Big(\bigcup_{i=1,\dots,\imax-1}\{a_i\not\leftrightarrow a_{i+1}\}\,\Big|\,F\Big)
		&\leq \sum_{i=1,\dots,\imax}
		\E\Big[\exp\Big\{-{W_{a_i}W_{a_{i+1}}}
			{|a_i-a_{i+1}|^{-\alpha}}\Big\}\,\Big|\,F\Big]\nonumber\\
		&\leq \sum_{i=1,\dots,\imax}\exp\big\{-|A_i|^{\frac {1-\delta}{\tau-1}}|A_{i+1}|^{\frac {1-\delta}{\tau-1}}2^{\alpha i}L^{-\alpha}\big\}
	\end{align}
	where for the last passage we have used that $|a_i-a_{i+1}|\leq L 2^{-i}$. Since $|A_i|$ is $2^{-i-1}L$ up to a unit and that $\delta<(2-\gamma)/2$, one can check that the exponent in the last display is bounded by
	\begin{align*}
	-|A_i|^{\frac {1-\delta}{\tau-1}}|A_{i+1}|^{\frac {1-\delta}{\tau-1}}2^{\alpha i}L^{-\alpha}
		\leq -c\, |A_i|^{\frac {2(1-\delta)-\gamma}{\tau-1}}
		\leq -c\, |A_{\imax}|^{\frac {2(1-\delta)-\gamma}{\tau-1}}
		\leq -c\, (\log L)^{M\frac {2(1-\delta)-\gamma}{\tau-1}}
	\end{align*}
	where $c>0$ is a constant not depending on $L$, possibly different of the one appearing in \eqref{dritto}. With this bound at hand, we conclude the estimate \eqref{ingrm} obtaining
	\begin{align}\label{rvsc}
	\P\Big(\bigcup_{i=1,\dots,\imax-1}\{a_i\not\leftrightarrow a_{i+1}\}\,\Big|\,F\Big)
		\leq \exp\big\{-c\, (\log L)^{M\frac {2(1-\delta)-\gamma}{\tau-1}}\big\}\,.
	\end{align}

Call now $\xi:=\min\{\delta, \frac {2(1-\delta)-\gamma}{\tau-1}\}$. The bounds \eqref{dritto} and \eqref{rvsc} yield
\begin{align*}
	\P\Big(\bigcup_{i=1,\dots,\imax-1}\{a_i\not\leftrightarrow 		a_{i+1}\}\,\Big|\,F\Big)
	&\leq \P\Big(\bigcup_{i=1,\dots,\imax-1}\{a_i\not\leftrightarrow a_{i+1}\}\,\Big|\,F\Big)
	+ \P(F^c)
	\leq 2\e^{-c(\log L)^{\xi M}}
\end{align*}
and this implies \eqref{unoelle} absorbing the factor $2$ in the constant $c$.

We notice that the bound \eqref{unoelle} also works if we replace $\mathcal D(1,L)$ by any $\mathcal D(x,y)$ with $x,y\in G^1$, since we could  repeat the whole argument here above with $x$ and $y$ replacing $1$ and $L$ and obtain an even better bound.
Hence, by a union bound, 
\begin{align*}
\P\big(\Delta_G>3(\log L)^M\big)
	&=\P\big(\exists j\in\{1,\dots,K\}\,,\, \exists x,y\in G^j\,:\, \mathcal D(x,y)>3(\log L)^M\big)\\
	&\leq KL^2 \P\big(\mathcal D(1,L)>3(\log L)^M\big)\\
	&\leq N^{\gamma+\varepsilon}\e^{-c( \log N)^{\xi M}}\,.
\end{align*}
The last quantity can be made summable in $N$ by choosing $M$ large enough so that, by the first Borel-Cantelli lemma, we are done. Notice that we ignored the fact that the graph induced on the $K$-th  the graph induced on $S_K=\{(K-1)L,\dots,N\}$ might be of a size larger than $L$; since this size cannot be larger than $2L-1$, though, the proof can be easily adapted.

\subsection{Proof of Proposition \ref{ratiodegrees}}
First of all we claim that, $\P$-a.s.~for all $N$ large enough,
\begin{align}\label{prot}
N
	\leq|E(G)|
	\leq N(\log N)^4.
\end{align}
The lower bound is obvious. 
We begin by bounding, for any $x\in \T_N$, and considering $N$ odd for simplicity,
\begin{align*}
\E[D_x\,|\,W_x] 
	& = 2+ 2\sum_{j=2}^{(N-1)/2} \E\big[1-\e^{-W_x W_{x+j}j^{-\alpha}}\,|\,W_x\big]\\
	&\leq 2+2\sum_{j=2}^{(N-1)/2}\Big( \E\Big[\tfrac{W_x W_{x+j}}{j^{\alpha}}\1{W_{x+j}<j^\alpha/W_x}\,|\,W_x\Big]
	+\P(W_{x+j}\geq \tfrac{ j^\alpha}{W_x}\,|\,W_x)\Big).
\end{align*}
By an elementary calculation, one can see that the $j$-th term of the sum equals $1$ if $j^\alpha<W_x$, while it is smaller than $c_1W_x^{\tau-1}/j^\gamma$ for some $c_1>0$ if $j^\alpha>W_x$ . 
Hence
\begin{align}\label{plpsi}
\E[D_x\,|\,W_x] 
	& \leq 2 W_x^{1/\alpha} +2c_1  \sum_{j=W_x^{1/\alpha}}^{(N-1)/2}\frac{W_x^{\tau-1}}{j^\gamma}
	\leq c_2 W_x^{1/\alpha}
\end{align}
where the last inequality can be checked by using the approximation of sums by definite integrals.
Notice that $D_x$ cannot be larger than $N$, so \eqref{plpsi} implies 
\begin{align}\label{stab}
\E[D_x\,|\,W_x] 
& \leq c_2 W_x^{1/\alpha}\wedge N\,.
\end{align}
Furthermore, for all $t>1$, it holds
\begin{align*}
\P(D_x >t \,\E[D_x\,|\,W_x]\,|\, W_x)
	\leq \e^{-ct}\,.
\end{align*}
The inequality can be checked via Bernstein's inequality, see Lemma \ref{lemmabernstein} with the $X_i$'s given by $\mathbbm 1_{\{x{\leftrightarrow}x+i\}}$  for $i=1,\dots, N-1$, taking  $M=1$ and noticing that  
$$
\sigma^2=
	\frac{1}{N-1}\sum_{i=1,\dots,N-1}\Var(\mathbbm 1_{\{x{\leftrightarrow}x+i\}}\,|\,W_x)\leq \frac{1}{N-1}\sum_{i=1,\dots,N-1}\E[\mathbbm 1_{\{x{\leftrightarrow}x+i\}}\,|\,W_x]=\E[D_x\,|\, W_x]\,.
$$
This ensures that
\begin{align*}
\P\big(\exists x:\,D_x>(\log N)^2 \E[D_x\,|\,W_x]\big)
	&\leq N\,\E\big[\P\big(D_x>(\log N)^2 \E[D_x\,|\,W_x]\,|\,W_x\big)\big]
	\leq N\e^{-c (\log N)^2},
\end{align*}
which is summable in $N$: we can conclude, thanks to the first Borel-Cantelli lemma and using \eqref{stab}, that almost surely for $N$ large enough
\begin{align}\label{hly}
D_x
	\leq \{c_2 W_x^{1/\alpha}\wedge N\} (\log N)^2\qquad \qquad\forall\, x\in\T_N\,.
\end{align}
We are now ready to  bound the total number of edges in $G$. Thanks to \eqref{hly} we get
\begin{align*}
\P(|E(G)|>N(\log N)^4)
	&\leq  \P\Big( \sum_{x\in \T_N}\{c_2 W_x^{1/\alpha}\wedge N\} >N(\log N)^2\Big)\,.
\end{align*}
We use once more Bernstein's inequality: we take in Lemma \ref{lemmabernstein} the independent variables $X_i=\{c_2 W_i^{1/\alpha}\wedge N\}$, the value $M=N$ and bound
\begin{align*}
\sigma^2
	\leq \E\big[c_2^2W_1^{2/\alpha}\wedge N^2\big]
	=\int_1^{N^2}\P\big(c_2^2W_1^{2/\alpha}>t\big)\,{\rm d}t
	\leq c_3 N^{2(1-\gamma/2)}\leq c_3 N\,.
\end{align*}
Bernstein's inequality then yields
\begin{align*}
\P(|E(G)|>N(\log N)^4)
	\leq \e^{-c(\log N)^2}\,.
\end{align*}
The last quantity is summable in $N$, so the Borel-Cantelli lemma finally gives the upper bound in  \eqref{prot}.

\smallskip

We turn our attention to $|E(\tGamma)|$. This is the number of edges in a SFP model with $K$ vertices and $\tgamma<1$. 
By item \ref{coscc:ii} in Proposition \ref{coscc} on the one hand, and by \eqref{corri} on the other, we know that, $\P$-a.s.~for all $N$ large enough,
\begin{align}\label{lprcs}
| E(\overline \tGamma)| \leq |E(\tGamma)|\leq |E(\overline{ \tGamma})| (\log K)^{c} \,.
\end{align}
where $\overline{\tGamma}$ indicates the simplified model described in Section \ref{simplifiedmodel}. 
At the same time, item \ref{coscc:v} in Proposition  \ref{coscc} tells us that 
\begin{align}\label{singrm}
K^{2-\tgamma}\lesssim 
|E(\overline{ \tGamma})|\lesssim K^{2-\tgamma}\,.
\end{align}
Since $2-\tgamma=1+\varepsilon/(4-2\gamma-2\varepsilon)$ and recalling that $K=N/L\geq  N^{2-\gamma-\varepsilon}$, \eqref{lprcs} and \eqref{singrm} yield, for some $c_4, c_5 >0$ that can be chosen independently of $\varepsilon$, 
\begin{align*}
N^{2-\gamma-c_4\varepsilon}
	\lesssim |E(\tGamma)|
	\lesssim N^{2-\gamma-c_5\varepsilon}\,.
\end{align*}
This, together with \eqref{prot}, implies~\eqref{pscl}.

\subsection{Proof of Proposition \ref{ratiopis}}
For simplicity in this proof we abbreviate $D_{S_j}$ for $D_{S_j}(G)$ and $\tD_j$ for $D_j(\tGamma)$.
We use the lower bound in \eqref{pscl} to see that 
\begin{align}\label{umpl}
\frac{\pi_{G}(S_j)}{\pi_{\tGamma}(j)}
	\leq N^{1-\gamma} \frac{D_{S_j}}{\tD_j}\qquad \qquad\forall {j\in\T_K}\,.
\end{align}
Therefore we want to show  that there exists $c>0$ such that, $\P$-a.s.~for $N$ large enough, for all ${j\in\T_K}$
\begin{align}\label{gcstl}
 \frac{D_{S_j}}{\tD_j}
	\lesssim  N^{\gamma-1+\varepsilon}
\end{align}
which together with \eqref{umpl} implies \eqref{gglne}.
We claim that, $\P$-a.s.~for all $N$ large enough, for all ${j\in\T_K}$ 
\begin{align}
D_{S_j}
	&\lesssim \Big\{L \vee \max_{x\in S_j}W_x^{1/\alpha}\Big\}\wedge N
\label{plnt}\\
\tD_j
&\gtrsim
\Big\{L^{-1}\max_{x\in S_j}W_x^{\tau-1} \vee 1\Big\}\wedge N^{2-\gamma-\varepsilon}\label{ose}
\end{align}
Before proving \eqref{plnt} and \eqref{ose} we show how to conclude by analyzing all possible cases. Fix $j$ and abbreviate $W:=\max_{x\in S_j}W_x$. Recall that $L=\lfloor N^{\gamma-1+\varepsilon}\rfloor$. We point out  that in principle some of the cases listed below might be empty, depending on the values of $\alpha$ and $\tau$.

\smallskip

\begin{enumerate}[label={\textbf{Case \arabic*.}}]
	\item $W>L^\alpha$. In this case we obtain 
	$
	D_{S_j}\lesssim W^{1/\alpha}\wedge N  \mbox{ and } \tD_j\gtrsim L^{-1}W^{\tau-1} \wedge N^{2-\gamma-\varepsilon}$. Here we distinguish two further sub-cases.
	
	\smallskip
	
		\begin{enumerate}[label*={\textbf{\alph*.}}]
		\item {$W>N^{1/(\tau-1)}$}. We have $D_{S_j}\lesssim N$ and  $\tD_j\gtrsim N^{2-\gamma-\varepsilon}$, so \eqref{gcstl} is verified.
		
		\smallskip
		
			\item $W\leq N^{1/(\tau-1)}$. We get $D_{S_j}\lesssim W^{1/\alpha}$ and $\tD_j\gtrsim L^{-1}W^{\tau-1}$, so that \eqref{gcstl} is  verified  since $W^{-(\tau-1)+1/\alpha}\leq 1$.
	\end{enumerate}

	\smallskip

	\item $W\leq L^\alpha$. Here we have 
	$D_{S_j}\lesssim L\leq  N^{\gamma-1+\varepsilon}$, while for \eqref{ose} we distinguish again sub-cases.
	
	\smallskip
	
	\begin{enumerate}[label*={\textbf{\alph*.}}]
		\item $W>L^{1/(\tau-1)}$. It holds  $\tD_j\gtrsim L^{-1}W^{\tau-1} \wedge N^{2-\gamma-\varepsilon}\geq N^{-(\gamma-1+\varepsilon)}W^{\tau-1} \wedge N^{2-\gamma-\varepsilon}$. If $W^{\tau-1}>N$, then $\tD_j\gtrsim N^{2-\gamma-\varepsilon}$, so that
		$D_{S_j}/\tD_j\lesssim N^{2\gamma-3+2\varepsilon}\leq N^{\gamma-1}$. 
		If instead $W^{\tau-1}\leq N$, then $\tD_j\gtrsim  N^{-(\gamma-1+\varepsilon)}W^{\tau-1}$, so that $D_{S_j}/\tD_j\lesssim N^{2(\gamma-1+\varepsilon)}W^{-(\tau-1)}\leq N^{2(\gamma-1+\varepsilon)}N^{-1}\leq N^{\gamma-1}$. So in both cases \eqref{gcstl} is verified.
		
		\smallskip
		
			\item $W\leq L^{1/(\tau-1)}$. In this case again $\tD_j\gtrsim N^{2-\gamma-\varepsilon}$ and \eqref{gcstl} holds.
	\end{enumerate}
\end{enumerate}
We move to the proof of \eqref{plnt} and \eqref{ose}. For equation 
\eqref{ose} we recall the simplified model described at the beginning of Section \ref{sec:gammino} and write $\overline\tD_j$ for the degree of node $j$ in the simplified model related to $\tGamma$. We have, using item \ref{coscc:ii} of Proposition \ref{coscc},
\begin{align*}
\tD_j 
	\geq \overline \tD_j
	\stackrel{\eqref{expecteddegreesimplified},\,\eqref{degreeconcentrationsimplified}}{\gtrsim }\E[\overline \tD_j\,|\,\tW_j]
	&\stackrel{\eqref{expecteddegreesimplified}}{= }
	(K-1)\cdot\Big\{K^{-\tgamma}\tW_j^{\tau-1}(\log K)^{-2(\tau-1)}\wedge 1\Big\}\\
	&\stackrel{\eqref{pfa}}{\gtrsim } K^{1-\tgamma}\big\{L^{-1}\max_{x\in S_j}W_x^{\tau-1}\vee 1\big\}\wedge K	
\end{align*}
which yields \eqref{ose} since $K^{1-\tgamma}\geq 1$ and recalling that $K=NL^{-1}\geq N^{2-\gamma-\varepsilon}$.

\smallskip

For equation \eqref{plnt}, we first of all notice that, by \eqref{hly},
\begin{align*}
D_{S_j}
	\lesssim \sum_{x\in S_j}  W_x^{1/\alpha}\,.
\end{align*}
Call 
$$
Y_j:=\sum_{x\in S_j}  W_x^{1/\alpha}\,,\qquad \quad M_j:=\max_{x\in S_j}W_x^{1/\alpha}.
$$  
We claim that there exists a constant $Q>0$ such that 
\begin{align}\label{pgna}
\P(\exists {j\in\T_K}\,:\,Y_j\geq Q (L \vee M_j))\;\,\mbox{{ is summable in $N$}. }
\end{align}
Thanks to the Borel--Cantelli lemma, this implies \eqref{plnt} by also noticing that $D_{S_j}\lesssim N$ by \eqref{prot}. Let us show \eqref{pgna}. 

We focus on $j=1$. Call $Y:=Y_1$ and $M:=M_1$ and let $\mu:=\E[W_x^{1/\alpha}]<\infty$. We distinguish between the cases where $M_j$ is smaller or larger than $L$. In the first case we can use directly the Fuk--Nagaev inequality \eqref{fuknagaev} with $y=L$ and $x=(Q-\mu)L$ to get that there exists a constant $c>0$ such that
\begin{align}\label{frnc}
\P(Y\geq Q (L \vee M)\,,\,M\leq L)
	\leq \Big(cL\frac{L^{1-\gamma}}{(Q-\mu)L}\Big)^{Q-\mu}
	\leq L^{-(Q-\mu)(\gamma-1)}\,,
\end{align}
where the last inequality holds for $Q$ large enough.

When instead the maximum exceeds $L$ we proceed as follows. First of all we divide the possible values of $M$ in intervals and bound
\begin{align}\label{crshrg}
\P(Y\geq  Q (L \vee M)\,,\,M> L)
	&\leq \sum_{\ell=1}^{\infty} \P\Big(Y\geq Q 2^{\ell-1}L\,,\,M\in (2^{\ell-1}L,2^\ell L]\Big)\,.
\end{align}
At the cost of a union bound we can suppose that $W_L$ is the largest $W_x$ in $S_1$, so that, for each $\ell$, the $\ell$-th summand in the last display can be dominated by
\begin{align}\label{dvnw}
L\, \P\Big(Y\geq  Q 2^{\ell-1}L\,&,\,W_L^{1/\alpha}\in (2^{\ell-1}L,2^\ell L]\,,\, M'\leq W_L^{1/\alpha}\leq 2^\ell L\Big)\nonumber\\
	&\leq L\,\P\big(W_L^{1/\alpha}\in(2^{\ell-1}L,2^\ell L] \big) \, \P\Big(Y'+2^\ell L\geq  Q 2^{\ell-1}L\,,\,M' \leq 2^\ell L\Big)
\end{align}
where $Y':=\sum_{x=1}^{L-1}W_x^{1/\alpha}$ and $M':=\max_{x=1,\dots,L-1}W_x^{1/\alpha}$. On the one hand,
\begin{align}\label{c2g}
\P\big(W_L^{1/\alpha}\in(2^{\ell-1}L,2^\ell L] \big)
	\leq 2^{-\gamma(\ell-1)} L^{-\gamma}
\end{align}
and on the other, using again the Fuk--Nagaev inequality \eqref{fuknagaev}, for $Q$ large enough,
\begin{align}\label{ejy}
\P\big(Y'+2^\ell L\geq Q 2^{\ell-1}L\,,\,M' \leq 2^\ell L\big)
	& \leq  \P\big(Y'\geq \mu L  +Q' 2^{\ell}L\,,\,M' \leq 2^\ell L\big)\nonumber\\
	&\leq \Big(c L \frac{(2^\ell L)^{1-\gamma}}{Q' 2^\ell L}\Big)^{Q'}
	\leq L^{-Q'(\gamma-1)}
\end{align}
where $Q'=\tfrac Q 2 -1 -\tfrac \mu 2$ is a constant that can be made arbitrarily large by taking $Q$ large. Summing up, using \eqref{dvnw}, \eqref{c2g} and \eqref{ejy} into \eqref{crshrg} we obtain
\begin{align*}
\P(Y\geq  Q (L \vee M)\,,\,M> L)
	\leq \sum_{\ell=1}^\infty L^{1-\gamma} 2^{-\gamma(\ell-1)}L^{-Q'(\gamma-1)}
	= 2 L^{-(Q'+1)(\gamma-1)}\,.
\end{align*}

This last expression together with \eqref{frnc} show that
\begin{align*}
\P(\exists j\in\{1,\dots,K\}\,:\,Y_j\geq Q(L\vee M_j))
	\leq K \P(Y\geq  Q (L \vee M))
	\leq K L^{-Q''}
\end{align*}
where $Q''$ is a constant that can be made arbitrarily large by taking $Q$ large. Since $L=\lfloor N^{\gamma-1+\varepsilon}\rfloor$ and $K=N/L$, by taking $Q$ large enough we can make \eqref{pgna} true, which in turn implies \eqref{plnt} as mentioned before. This concludes the proof of Proposition \ref{ratiopis}.\qed

\subsection{Proof of Proposition \ref{ags}}
Fix any ${j\in\T_K}$ and call $z^*$ the point in $S_j$ realizing the maximum of $\pi_{G}(z)$ in $S_j$. We notice that
\begin{align*}
\sum_{z\neq w \in S_j}\pi_{G}(w)\pi_{G}(z)
	&=\pi_{G}(z^*)\sum_{w\neq z^*}\pi_{G}(w)+\sum_{z\neq z^*}\pi_{G}(z)\sum_{w\neq z}\pi_{G}(w)\\
	&\leq 2\pi_{G}(S_j)\big(\pi_{G}(S_j)-\pi_{G}(z^*)\big)\,.
\end{align*}
By using that $|E(G)|\geq N$, the probability on the left-hand side~of \eqref{gglne2} can be therefore upper bounded using also a union bound by
\begin{align*}
\mbox{l.h.s. of~\eqref{gglne2}}
	\leq K\, \P\big(D_{S_1}\big(D_{S_1}-D^{*}_{S_1}\big)>N^{\gamma}\big) 
\end{align*}
with $D^{*}_{S_1}:=\max_{x\in S_1}D_x$. We call
\begin{align*}
Y:=\sum_{x\in S_1}W_x^{1/\alpha},\qquad  \qquad M:=\max_{x\in S_1}W_x^{1/\alpha}.
\end{align*}
Equation \eqref{plnt} in the proof of Proposition \ref{ratiopis} and equation \eqref{hly} imply that there exists a $c>0$ such that, for $N$ large enough, $\P$-a.s.,
\begin{align*}
D_{S_1}
	\leq (\log N)^{c/2} \big\{M\vee L\big\}\qquad
		\mbox{and}\qquad
	D_{S_1}-D^{*}_{S_1}\leq (\log N)^{c/2}(Y-M)\,.
\end{align*}
Therefore, in order to show that the left-hand side~of \eqref{gglne2} tends to zero, it will be enough to prove
\begin{align}\label{dea}
\lim_{N\to\infty}K\, \P\big(\big\{M\vee L\big\}(Y-M)>N^{\gamma}(\log N)^{-c}\big)
	=0\,.
\end{align}
We bound this probability depending on the value of $M$.
\begin{description}[leftmargin=*]
\item[Case $M<L(\log N)^c$]
first of all we use that $\big\{M\vee L\big\}<L(\log N)^c$ to see that
\begin{align*}
\P\big(\big\{M\vee L\big\}(Y-M)>\tfrac{N^{\gamma}}{(\log N)^{c}}\,,\, M<L(\log N)^c\big)
	\leq \P\big(Y>\tfrac{N^{\gamma}}{L(\log N)^{2c}}\,,\, M<L(\log N)^c\big)\,.
\end{align*}
Then, since $\mu:=\E[W_x^{1/\alpha}]<\infty$ and since $N^{\gamma}L^{-1}(\log N)^{-2c}\geq 2\mu L$ 
by \eqref{miktti}, one obtains
\begin{align*}
\P\big(Y>\tfrac{N^{\gamma}}{L(\log N)^{2c}}\,,\, M<L(\log N)^c\big)
	\leq \P\big(Y>\tfrac{N^{\gamma}}{2L(\log N)^{2c}}+\mu L\,,\, M<L(\log N)^c\big)\,.
\end{align*}
Using the Fuk--Nagaev inequality \eqref{fuknagaev} we bound the right hand side of the last display with
\begin{align*}
\Big(c_1 L\frac{(L(\log N)^c)^{1-\gamma}}{N^\gamma L^{-1}(\log N)^{-2c}}\Big)^{\tfrac{N^\gamma}{2L(\log N)^{2c}}\tfrac{1}{L(\log N)^c}}
	\leq \Big((\log N)^{c_2}L^{3-\gamma}N^{-\gamma}\Big)^{N^{\gamma}L^{-2}(\log N)^{-4c}}
\end{align*}
for some $c_1,c_2>0$.
Since there exists $\delta>0$ such that $L^{3-\gamma} N^{-\gamma}<N^{-2\delta}$ and $N^\gamma L^{-2}>N^{2\delta}$, for $N$ large enough we can bound
\begin{align}\label{plut}
\P\big(\big\{M\vee L\big\}(Y-M)>\tfrac{N^{\gamma}}{(\log N)^{c}}\,,\, M<L(\log N)^c\big)
	\leq N^{-\delta N^{\delta}}\,.
\end{align} 
\item[Case $M\ge N^{1-2\varepsilon}$] when the maximum is large, one can simply bound
\begin{align}\label{ppo}
\P\big(\big\{M\vee L\big\}(Y-M)>\tfrac{N^{\gamma}}{(\log N)^{c}}\,,\, M\geq N^{1{-2}\varepsilon}\big)
	&\leq \P\big(M\geq N^{1{-2}\varepsilon}\big)
	\leq N^{-1+2\varepsilon \gamma}\,.
\end{align}
\item[Case $M\in[L(\log N)^c,\,N^{1{-2}\varepsilon})$]
in this case $\{M\vee L\}=M$ and we divide the possible values of $M$ in intervals. Let $\ell_{\min}:=c\log_2\log N$ and $\ell_{\max}:=(2-\gamma{-3\varepsilon})\log_2N$ (here and below we are ignoring the fact that $\ell_{\min}$ and $\ell_{\max}$ might not be integers, in which case we could just take their integer part). We partition
\begin{align*}
\P\big(M(Y-M)>&\tfrac{N^{\gamma}}{(\log N)^{c}}\,,\, M\in[L(\log N)^c,\,N^{1{-2}\varepsilon})\big)
	= \sum_{\ell=\ell_{\min}+1}^{\ell_{\max}}
	\P\big(Y-M>\tfrac{N^{\gamma}}{M(\log N)^{c}}\,,\, M\in I_\ell\big)
\end{align*}
where  $I_\ell:=[L2^{\ell-1},L2^{\ell})$. We bound now the  probability in the sum.
At the cost of a union bound we can suppose that $W_L$ is the largest $W_x$ in $S_1$. Calling $Y':=\sum_{x=1}^{L-1}W_x^{1/\alpha}$, it holds 
\begin{align}\label{cpd}
\P\big(Y-M>\tfrac{N^{\gamma}}{M(\log N)^{c}}\,,\, M\in I_\ell\big)
	&\leq L\,\P\big(Y-M>\tfrac{N^{\gamma}}{M(\log N)^{c}}\,,\, M=W_L^{1/\alpha}\in I_\ell\big)\nonumber\\
	&\leq L\,\P(W_L^{1/\alpha}\in I_\ell)\P\big(Y'>\tfrac{N^{\gamma}}{L2^\ell(\log N)^{c}}\big)\,.
\end{align}
By \eqref{pgna} and the first Borel-Cantelli lemma, we know that with probability $1$ there exists a $Q>0$ such that, for all $N$ large enough, $Y'<Y<Q(L\vee M)$. 
Therefore, keeping into account that $N^{\gamma}L^{-1}2^{-\ell}(\log N)^{-c}Q^{-1}> L$ for all $\ell\in[\ell_{\min},\ell_{\max}]$, for $N$ large enough
\begin{align*}
\P\big(Y'>\tfrac{N^{\gamma}}{L2^\ell(\log N)^{c}}\big)
	&\leq \P(M>\tfrac{N^{\gamma}}{QL2^\ell(\log N)^{c}})
	\lesssim N^{-\gamma^2}L^{1+\gamma}2^{\gamma {\ell}}\,.
\end{align*}
Plugging this into \eqref{cpd} and using \eqref{c2g} yields
\begin{align*}
\P\big(Y-M>\tfrac{N^{\gamma}}{M(\log N)^{c}}\,,\, M\in I_\ell\big)
	\lesssim N^{-\gamma^2}L^2
	\leq N^{-\gamma^2+2\gamma-2+2\varepsilon}
\end{align*}
for all $\ell_{\min}\leq \ell\leq \ell_{\max}$. 
Since the number of $\ell$'s is logarithmic, we conclude that
\begin{align}\label{doduc}
\P\big(\big\{M\vee L\big\}(Y-M)>&\tfrac{N^{\gamma}}{(\log N)^{c}}\,,\, M\in[L(\log N)^c,\,N^{1{-2}\varepsilon}\big)
	\lesssim N^{-\gamma^2+2\gamma-2+2\varepsilon}.
\end{align}
\end{description}

We gather the results in \eqref{plut}, \eqref{ppo} and \eqref{doduc} to see that
\begin{align*}
K\, \P\big(\big\{M\vee L\big\}(Y-M)>N^{\gamma}(\log N)^{-c}\big)
	&\lesssim N^{2-\gamma-\varepsilon}(N^{-\delta N^\delta}+N^{-1+2\varepsilon\gamma}+N^{-\gamma^2+2\gamma-2+2\varepsilon})\\
	&\lesssim N^{-(\gamma-1)+3\varepsilon}
\end{align*}
which proves \eqref{dea} for $\varepsilon$ small enough.\qed

\section{Case {$\alpha>2,\,\gamma>2$}: upper bound of Theorem {\ref{brng}}}\label{tbls}

For the reader's convenience we restate the upper bound of Theorem \ref{brng} as a Proposition:
\begin{proposition}\label{prop:gammaalphagrande}
	Let $\alpha>2$ and $\gamma>2$. There exists $c>0$ such that, $\P$-a.s.~for all $N$ large enough, we have
	\begin{align*}
	\tmix(G_N)\leq N^2(\log N)^c\,.
	\end{align*}
\end{proposition}
Before proving the Proposition, we need two lemmas.
\begin{lemma}\label{GRA}
 Let $\alpha>2$ and $\gamma>2$. There exists $c>0$ such that we have
 \begin{equation}\label{Exp_Dtot}
  \E[D_{G_N}]{\leq cN}
 \end{equation}
and
 \[
  \Var(D_{G_N}){\leq cN}.
 \]
\end{lemma}
\begin{proof}
 In the regime $\alpha>2$ and $\gamma>2$ we can use~\citet[Theorem~2.2]{deijfen} to infer that the node degrees have a variance which is bounded in $N$ and so, in particular, they have a bounded mean. In fact, it is easy to get convinced that both $\E[D_{G_N}]$ and $\Var(D_{G_N})$ are dominated by their infinite counterpart on $\Z$ since the random variables $\{\1{x\leftrightarrow y}:\,x,\,y\in {\T_N}\}$ are positively correlated. The result \eqref{Exp_Dtot} on the expectation of $D_{G_N}$ immediately follows.
  
  \smallskip
  
As for the variance, the first step is the simple observation that 
\begin{align}\label{part_long_2}
\E\big[D_{G_N}^2\big]
 	&=N \E\big[D_1^2\big]+\sum_{x\neq y\in \T_N}\E\big[D_x D_y\big]
 	=cN+\sum_{x\neq y\in \T_N}\E\big[D_x D_y\big]
\end{align}
for some constant $c>0$.
Note that here we use the fact that the degrees have bounded variance. Let us look, for $x\neq y$, at
\begin{align}
 \E\big[D_x D_y\big]
 	=&\,\E\Bigg[\Big(\1{x\leftrightarrow y}+\sum_{z\neq x,\,y}\1{x\leftrightarrow z}\Big)\Big(  \1{x\leftrightarrow y}+\sum_{w\neq x,\,y}\1{y\leftrightarrow w}\Big)\Bigg]\nonumber\\
 =&\sum_{z\neq x,\,y}\sum_{w\neq x,\,y,\,z}\P(x\leftrightarrow z)\P(y\leftrightarrow w)+\sum_{z\neq x,\,y}\E\left[\1{x\leftrightarrow z}\1{y\leftrightarrow z}\right]
 \nonumber\\
 &+\sum_{z\neq x,\,y}\E\left[\1{x\leftrightarrow y}\left(\1{x\leftrightarrow z}+\1{y\leftrightarrow z}\right)\right]+\E\big[\1{x\leftrightarrow y}^2\big]\,.\label{sound}
\end{align}
The first double sum in the last expression is upper bounded by
$
\E[D_x]\E[D_y]=\sum_{z\neq x}\P(x\leftrightarrow z)\sum_{w\neq y}\P(y\leftrightarrow w)$, and therefore
\begin{align}\label{mcostn}
 \E\big[D_x D_y\big]
 	\leq \E[D_1]^2 +\mathcal R_{x,y}
\end{align}
with $\mathcal R_{x,y}$ given by
\begin{align}\label{sum_R}
	\sum_{z\neq x,\,y}\E\left[\1{x\leftrightarrow z}\1{y\leftrightarrow z}\right]+\sum_{z\neq x,\,y}\E\left[\1{x\leftrightarrow y}\1{x\leftrightarrow z}\right]+\sum_{z\neq x,\,y}\E\left[\1{x\leftrightarrow y}\1{y\leftrightarrow z}\right]
	+\P(x\leftrightarrow y)\,.
\end{align}

We will now provide a bound for $\cR_{x,y}$. Using the inequality $\P(A\cap B)\le \P(A)\wedge \P(B)$ for any two events $A$ and $B$ and the fact that the weights have finite mean, we can see that
\begin{align*}
 \E\left[\1{x\leftrightarrow z}\1{y\leftrightarrow z}\right]
 	\le \P(x\leftrightarrow z)\wedge \P(y\leftrightarrow z)
 	\stackrel{\eqref{gpand}}{\le} 
 	\|x-z\|^{-2-\varepsilon}\wedge \|y-z\|^{-2-\varepsilon} 
\end{align*}
for some $\varepsilon>0$ such that $\gamma,\alpha>2+\varepsilon$.
Noticing that $\|x-z\|\vee \|y-z\|\geq \big\lceil{\|x-y\|}/{2}\big\rceil$ for all $z\in \T_N$, we bound
\begin{align}
\sum_{z\neq x,\,y}&\E\left[\1{x\leftrightarrow z}\1{y\leftrightarrow z}\right] 
	\le {2 \sum_{\ell=\big\lceil\frac{\|x-y\|}{2}\big\rceil}^\infty \ell^{-2-\varepsilon}
\leq 
c\,{\|x-y\|}^{-1-\varepsilon}}\label{lengthy_eq}
%
\end{align}
for some $c>0$.
Analogously, using that $\E[\1{x\leftrightarrow y}\1{x\leftrightarrow z}] \leq \P(x\leftrightarrow y)$ when $\|x-z\|\leq \|x-y\|$ while $\E[\1{x\leftrightarrow y}\1{x\leftrightarrow z}] \leq \P(x\leftrightarrow z)$ when $\|x-z\|\geq \|x-y\|$, we bound
\begin{align}\label{kib}
 \sum_{z\neq x,\,y}\E\left[\1{x\leftrightarrow y}\1{x\leftrightarrow z}\right] 
 	&\leq 2\|x-y\| \cdot\|x-y\|^{-2-\varepsilon} +2\sum_{\ell=\lceil \|x-y\|\rceil}\ell^{-2-\varepsilon}\nonumber\\
 	&\leq c\|x-y\|^{-1-\varepsilon}
\end{align}
for some $c>0$. 
We can bound similarly the third term of~\eqref{sum_R}. Again~\eqref{gpand} entails that the last term of~\eqref{sum_R} is upper-bounded by $\|x-y\|^{-2-\varepsilon}$. This estimate, \eqref{lengthy_eq} and \eqref{kib} plugged back into~\eqref{sum_R} yield $\cR_{x,y}\leq c\|x-y\|^{-1-\varepsilon}$ for some constant $c>0$.
This fact and \eqref{mcostn} bring to
\begin{align}\label{plo}
\sum_{x\neq y\in \T_N}\E[D_x D_y]
	\leq \sum_{x\neq y\in \T_N}\E[D_1]^2+\|x-y\|^{-1-\varepsilon}
	\leq N^2\E[D_1]^2 + c N\,.
\end{align}
for some $c>0$, where we have used the fact that $\sum_{n=1}^\infty n^{-1-\varepsilon}<\infty$. 

We can conclude with \eqref{part_long_2} and \eqref{plo} that, for some $c_1>0$,
\begin{align*}
\Var(D_{G_N})
	\leq N^2\E[D_1]^2 + c_1 N  - \E[D_{G_N}]^2
	= c_1N.  \tag*{\qedhere} 
\end{align*}
\end{proof}
For our purpose, not only do we need the mean and variance of the total degree, but also a concentration result.
\begin{lemma}\label{conc_Dtot}
For $\alpha>2, \,\gamma>2$ we have that, $\P$-a.s.~for all $N$ large enough,
\begin{align*}
\big|D_{G_N}-\E[D_{G_N}]\big|\leq  {N}\log N\,.
\end{align*}
\end{lemma}
\begin{proof}
 By Chebyshev's inequality
\begin{equation}\label{taac}
 \P\left(|D_{G_N}-\E[D_{G_N}]|\ge {N}\log N\right)\le \frac{\Var(D_{G_N})}{N^2(\log N)^2}\,.
\end{equation}
The statement is a consequence of Lemma~\ref{GRA} and the first Borel--Cantelli lemma.
\end{proof}

\begin{proof}[Proof of Proposition~\ref{prop:gammaalphagrande}]First of all, we recall the upper bound~\cite[Remark~10.17]{levinperes}
\begin{equation}\label{tmix_thit}
\tmix(G)\le 4 t_{\mathrm{hit}(G)}+1
\end{equation}
for any irreducible 
chain on a graph $G$, where $t_{\mathrm{hit}}(G)=\max_{x,\,y\in G}E_x[\tau_y]$ is the mean hitting time $\tau_y$ of $y$ of the chain starting at $x$. Using for example \citet[Proposition 10.7]{levinperes} one shows that the maximum hitting time for a graph with $N$ vertices and $M$ edges is at most of order $MN$. But by Lemmas~\ref{GRA} and~\ref{conc_Dtot} we know that $\P$--a.s.~for all $N$ large enough we have $D_{G_N}\leq 2N\log N$. Since $D_{G_N}$ is proportional to the number of edges, we obtain that there exists $c>0$ such that, $\P$-a.s.~for all $N$ large enough,
\[
 t_{\mathrm{hit}}(G_N) \leq cN^2\log N \,.
\]
This, together with \eqref{tmix_thit}, concludes the argument.
\end{proof}

\section{Lower bounds}\label{lowerbounds}
In this section we will prove a lower bound on the mixing time for all regimes. In Subsection \ref{llt} we deal with the case $\alpha\in (1,\,2),\, \tau>2$ and with the case $\tau\in(1,\,2),\,\gamma\in(1,\,2)$ at once. In Subsection \ref{ven} we treat the case $\gamma>2,\,\alpha>2$.

\subsection{Lower bound on {$\tmix(G_N)$} for case {$\alpha\in (1,\,2),\,\tau>2$} and case {$\tau\in(1,\,2),\,\gamma\in(1,\,2)$}}\label{llt}
We recall the desired result in the following proposition.
\begin{proposition}\label{boundary_boundal}
 Consider either the case $\alpha\in (1,\,2)$ and $\tau>2$ or the case $\tau\in(1,\,2)$ and $\gamma\in(1,\,2)$. Then it holds
 \begin{align*}
 	\P\big(\tmix(G_N)< (\log N)^{-3}N^{(\alpha\wedge \gamma)-1}\big)
 		\xrightarrow{N\to\infty}0\,.
 \end{align*}
\end{proposition}
\begin{proof}
Consider the set
\[
S:=\left\{1,\,\ldots,\,\left\lfloor N/2\right\rfloor\right\}\,.
\]
Using for example the approximation of sums by definite integrals, one can check that
\begin{align}\label{eslp}
\E[|\partial S|]
=\sum_{x\in S,\,y\notin S}\P(x\leftrightarrow y)
&\stackrel{\eqref{gpand}}{\leq} (\log N)^2
\sum_{x=1}^{\left\lfloor N/2\right\rfloor}\sum_{y=\left\lfloor N/2\right\rfloor+1}^N\|x-y\|^{-(\alpha\wedge \gamma)}\\
&\leq c(\log N)^2 N^{2-(\alpha\wedge \gamma)}
\end{align}
for some constant $c>0$  depending on $\alpha$ and $\tau$.
Recall the definitions of the bottleneck ratio \eqref{cmic} and of the Cheeger constant $\Phi_\ast$ in \eqref{cheeger} and its relation with the mixing time in \eqref{tmixcheeger}. Using \eqref{eslp} and Markov's inequality we obtain, with $c>0$ a constant depending on $\alpha$ and $\tau$ that might change from line to line,
\begin{align*}
\P(\tmix(G_N)< (\log N)^{-3}N^{(\alpha\wedge \gamma)-1})
	&\leq \P(\Phi_\ast> c(\log N)^3 N^{1-(\alpha\wedge \gamma)})\\
	&\leq \P(|\partial S|>c(\log N)^3 N^{2-(\alpha\wedge \gamma)})\\
	&\leq c\E[|\partial S|](\log N)^{-3} N^{-2+(\alpha\wedge \gamma)})\xrightarrow{N\to\infty}0\,.
\end{align*}
\end{proof}
\subsection{Lower bound on {$\tmix(G_N)$} for case {$\gamma>2,\,\alpha>2$}}\label{ven}
For a moment, let us consider $\S_N=\{1,2,\dots, N\}$ as a segment rather than a torus, so that in particular the distance between $x,y\in \S_N$ is just $|x-y|$ rather than the torus distance $\|x-y\|$. For a graph $G$ on $\S_N$, we say that a point $x\in\S_N$ is a {\it cut-point} for $G$ if there is no edge between a point in $\{1,2,\dots,x-1\}$ and a point in $\{x+1,\dots,N\}$. We say that $x\in \S$ is a {\it good cut-point} if $x-1$, $x$ and $x+1$ are cut-points. 

\smallskip

Consider now a SFP random graph $G_N(\S)$ on $\S_N$, that is, the model described in Section \ref{shesshion} and with link probability given by equation \eqref{eq:con_proba} with $|x-y|$ replacing $\|x-y\|$.
In the following lemma, we will show there is a positive density of good cut-points for $N$ large.
\begin{lemma}\label{lem:good_cut}
 Consider a SFP random graph $G_N(\S_N)$ on the segment $\S_N$. There exists $c>0$ such that 
 \[
 \lim_{N\to\infty}\P\big(\left|\left\{\text{Good cut-points in }\S_N\right\}\right|\ge cN\big)=1\,.
 \]
\end{lemma}
\begin{proof}
We want to make use of the ergodic theorem, so we start to investigate the infinite SFP random graph on $\Z$. 
The infinite graph $G_N(\Z)$ can be constructed under the measure $\mu:=\bigotimes_{x\in\Z} Law(W_x)\otimes \bigotimes_{x,\,y\in \Z}Law(U_{x,\,y})$, where $U_{x,\,y}$ are i.i.d.~uniform random variables on $[0,1]$. To do so, first sample the weight $W_x$ of each point in $\Z$, then connect each couple of points $x$ and $y$ with $|x-y|>1$ with an edge if $U_{x,y}\leq \P(x\leftrightarrow y\,|\,W_x,W_y)$ (cfr.~\eqref{eq:con_proba} with $|\cdot|$ replacing $\|\cdot\|$) and connect all nearest neighbours as usual. The definition of cut-point and of good cut-point on the infinite graph are just the same we gave for the segment before the lemma. We compute, using Jensen's inequality,
\begin{align*}
\mu\left(x_0\text{ is a good cut-point of $G_N(\Z)$}\right)
&=\mu(x_0,\,x_0-1,\,x_0+1\text{ are cut-points in $G_N(\Z)$})\\
&=\mathbb{E}_\mu\Big[\prod_{x\le x_0-1\atop y\ge x_0+1}\e^{-\frac{W_x W_y}{\|x-y\|^\alpha}}\Big]
\ge \exp\Big(-\sum_{x\le x_0-1\atop y\ge x_0+1}\frac{\E[W_0]^2 }{\|x-y\|^\alpha}\Big)\,.
\end{align*}
Note that, since $\alpha>2$,
\begin{equation}\label{eq:alpha_summable}
\sum_{x\le x_0-1\atop y\ge x_0+1}{\|x-y\|^{-\alpha}}
	<\infty\,,
\end{equation}
from which we deduce that
\[
\mu\left(x_0\text{ is a good cut-point of $G_N(\Z)$}\right)=:\psi(\alpha,\tau)>0\,.
\]
Measure $\mu$ is invariant under the shift $x\mapsto x+1$ in $\Z$. Call $M_N$ the number of good cut-points in $\{1,2,\dots,N\}$ in $G_N(\Z)$. 
By the ergodic theorem
\[
\lim_{N\to\infty}\1{\tfrac{M_N}N\ge \tfrac{\psi(\alpha,\,\tau)}{2}}=1 \quad \mu-a.s.
\]
Therefore 
\begin{align}\label{liliac}
 \lim_{N\to\infty}\mu\left(\frac{M_N}{N}\ge \frac{\psi(\alpha,\,\tau)}{2}\right)=1
\end{align}
by the dominated convergence theorem. 

We observe now that an instance of scale-free percolation on $\S_N$ can be obtained as follows: construct $G_N(\Z)$ and then restrict it to the vertices in $\S_N$. Using this construction we see that if a point is a good cut-point in the infinite graph, then it is so also on its restriction to $\S_N$.
The lemma follows then automatically from \eqref{liliac}.
\end{proof}
\begin{remark}
	From Lemma \ref{lem:good_cut} it is possible to deduce that, in the regime where $\alpha>2$ and $\gamma>2$, there exists $c>0$ such that
	\[
	\lim_{N\to\infty}\P\left({\rm diam}(G_N)\le cN\right)=0\,,
	\]
	that is, the diameter of $G_N$ is linear with probability tending to $1$.
\end{remark}
We are now ready to prove the lower bound for the regime under consideration.
\begin{proposition}\label{prop:LBag22}
Let $\alpha>2$ and $\gamma>2$. There exists $c>0$ such that 
\[
\lim_{N\to\infty}\P(\tmix(G_N)\ge c N^2)=1\,.
\]
\end{proposition}
\begin{proof}[Sketch of the proof]
The proof follows the argument in~\citet[Proposition 4.1]{BBY08}, and for  completeness we will sketch here the main points. %
For simplicity we assume $N$ to be divisible by $8$. We partition $\T_N$ into three sets:
\[
 A:=\left\{1,\,2,\,\ldots,\,\tfrac{N}{2}\right\},\;B:=\left\{\tfrac{N}{2}+1,\,\ldots,\,\tfrac{3N}{4}\right\},\;C:=\left\{\tfrac{3N}{4}+1,\,\ldots,\,N\right\}.
\]
Also call
\[
 K_i:=\left\{(i-1)\tfrac{N}{8}+1,\,\ldots,\,i\tfrac{N}{8}\right\} \qquad i=1,\,2,\,\ldots,\,8.
\]
By Lemma~\ref{lem:good_cut} we can assume that there exists a constant $c>0$ for which, with probability arbitrarily close to 1,  $G_N$ restricted to the segment $K_i$ contains at least $c N$ good cut-points. We can also assume without loss of generality (eventually rotating the cycle) that $\pi(A)\ge \pi(B\cup C)$, $\pi(B)\ge \pi(C)$. This entails that $\pi(A\cup B)\ge 3/4$. For $x\notin A\cup B$, let $T$ be the hitting time of $A\cup B$ for the simple random walk $(X_n)_{n\in \N_0}$ on $G_N$ starting at $x$. We denote its law (given the realization of the graph $G_N$) as $\pp_x$ and the relative expectation as $\ep_x$. For $n\ge \tmix(G_N)$ one obtains
\[
 \frac34-\pp_x\left(S_n\in A\cup B\right)\le\sum_{y\in A\cup B}\left|\pi(y)-P^n(x,\,y)\right|\le 2 \|\pi-P^n(x,\,\cdot)\|_{ {}_{\rm TV}}\le {\frac{1}{2}}.
\]
Hence, for any $x\notin A\cup B$ and $n\ge \tmix(G_N)$, one has $\pp_x(T\le n)\geq \pp_x(S_n\in A\cup B)\geq 1/4$. It follows that for any real $s\geq 0$ one has $\pp_x(T>s)\le (3/4)^{s/\tmix(G_N)-1}$ and thus there exists $c>0$ independent of $N$ such that
\begin{equation}\label{eq:large_hit}
  E_x[T]\le c\tmix(G_N).
 \end{equation}
Call $u:=7N/8$. Using the language of electrical networks~\cite[Chapter 9]{levinperes} we ground the set $A\cup B$ and set a potential in $u$ so that there is a unit of current flowing from $A\cup B$ to $u$. Let $x_1,\,\ldots,\,x_{cN}$ be the good cut points on the side of $u$ with at least $1/2$ the current. In particular, each of them will be crossed by at least half the current. In turn, using the relation between voltage and resistance and that the number of cut edges (resistors of resistance 1 connected in series) is at least the number of good cut points, the voltage $v(x_i)$ at $x_i$ is at least $i/2$. Hence
\[
 \ep_u[T]
 	\ge \sum_{x\notin A\cup B\atop x\text{ good cut point}} \ep_u\Big[\sum_{i=0}^T\1{S_i=x}\Big]
 	=\sum_{x\notin A\cup B\atop x\text{ good cut-point}} D_x\, v(x)
 	\ge \sum_{i=1}^{cN}i=c_1 N^2
\]
for some $c_1>0$.
By~\eqref{eq:large_hit} we derive that there exists a $c>0$ such that $\tmix\ge c N^2,$ thus giving the desired result.
\end{proof}

\appendix
\section{}

In this first part of the appendix we recall some well-known inequalities and also prove a concentration inequality that we will use in different parts of the paper.

\begin{lemma}[Bernstein inequality]\label{lemmabernstein}
	Let $X_i,\,i=1,\,\ldots,\,N$, be independent centered random variables with $|X_i|\le M$ for all $i$ a.s., and define $\sigma^2:=\sum_{i=1}^N \Var(X_i)/N$. For all $u>0$
	\begin{equation}\label{bernstein}
	\P\Big(\sum_{i=1}^N X_i \ge u\Big)
		\le\exp\Big(-\frac{u^2}{2(N\sigma^2+Mu/3)}\Big)\,.
	\end{equation}
\end{lemma}

\begin{theorem}[Fuk--Nagaev inequality{~\cite[Theorem 1.2]{Nag79},~\cite[Theorem 5.1 (ii)]{quentinnotes}\footnote{We were notified by Quentin Berger of a typo in the statement of Theorem 5.1 (ii). We are reporting here the correct statement.}}]\label{thm:fuknagaev}
	Let $X_1,\dots,X_N$ be i.i.d.~random variables such that \begin{align*}
	\P(X_1>t)=t^{-\gamma}
	\end{align*}
	for some $\gamma>1$ and let $\mu:=\E[X_1]$. Call $S_N:=\sum_{i=1,\dots,N}X_i$ and $M_N:=\max_{i=1,\dots,N}X_i$. Then there exists a  constant $c>0$ such that, for all $y\leq x$,
	\begin{align}\label{fuknagaev}
	\P(S_N-N\mu\geq x\,,\,M_N\leq y)\leq \Big(cN y^{1-\gamma}x^{-1}\Big)^{x/y}\,.
	\end{align} 
	
\end{theorem}

\begin{proposition}\label{prop:beb}
	For $N\in\N$, let $\{Z_{N,x}\}_{x=1,\dots,N}$ be a collection of $N$ positive independent random variables. Assume that there exists a random variable $A_N$ such that $Z_{N,x}\leq A_N<\infty\,$ a.s.~and such that
	\begin{align}\label{beb}
	\sum_{x=1}^{N}\E[Z_{N,x}^2]
		\ge{A_N^2(\log N)^2}\,.
	\end{align}
	Call  
	$Z_N:=\sum_{x=1}^NZ_{N,x}$ and $U_N:= \sqrt{\sum_{x=1}^{N}\E[Z_{N,x}^2]}\log N$. 
	For a constant $c>0$ not depending on $N$ one has
	\begin{align}\label{dun}
	\P(|Z_N-\E[Z_N]|>U_N)
		&\leq \e^{-c(\log N)^2}
	\end{align}
	so that, $\P$--a.s.~for all $N$ large enough,
	\begin{align}\label{newyork}
	\big|Z_{N}-\E[Z_{N}] \big|
		\leq \sqrt{\mbox{$\sum_{x=1}^{N}$}\E[Z_{N,x}^2]}\log N\,.
	\end{align}
\end{proposition}
\begin{proof}
	We compute
	\begin{align*}
	\P(Z_N>\E[Z_N]+U_N)
		&\stackrel{\eqref{bernstein}}{\leq} \exp\Big\{-\frac{U_N^2/2}{\sum_{x=0}^{N-1}\E[Z_{N,x}^2]+A_N U_N /3}\Big\}.
	\end{align*}
	Condition~\eqref{beb} implies that the last exponent is smaller than $-\tfrac 38(\log N)^2$. With the very same method a bound of the same order can be obtained for the probability of the event $\{Z_N<\E[Z_N]-U_N\}$, giving \eqref{dun}. It follows that
	\begin{align*}
	\sum_{N=1}^{\infty} \P(\big|Z_{N}-\E[Z_{N}] \big|
	\leq U_N)<\infty\,,
	\end{align*}
	and we can conclude with the first Borel-Cantelli lemma.
\end{proof}

\section{}
In this part of the Appendix we collect the technical proofs of some properties of our models.

\subsection{Proof of Proposition \ref{coscc}}\label{appe}
	For~\ref{coscc:i}, we notice that $\{\exists\, x\in\T_N:\,W_x\geq N^\alpha(\log N)^2\}\subseteq\{\sG_N\,\mbox{is fully connected}\}$ and calculate
\begin{align*}
\P(\not\exists x\in\T_N:\,W_x\geq N^\alpha(\log N)^2)
	=\P(W_1< N^\alpha(\log N)^2)^N
	\leq \e^{-N^\varepsilon}\,,
\end{align*}
for some $0<\varepsilon<1-\gamma\,$ that does not depend on $N$. Since the right-hand side~is  summable in $N$, the first Borel--Cantelli lemma implies the claim. 

\smallskip

To show \ref{coscc:ii} we uniformly bound the probability that an adge between any $x$ and $y$ is present in $\sG_N$ but not in $G_N$:
\begin{align*}
\P(x\stackrel{\sG_N }{\leftrightarrow}y \cap x\stackrel{G_N }{\not\leftrightarrow}y)
	\leq \P(x\stackrel{G_N }{\not\leftrightarrow}y\,|\, W_xW_y\geq N^\alpha(\log N)^2) 
	\leq \e^{-N^\alpha(\log N)^2/{\|x-y\|}^\alpha}
	\leq  \e^{-(\log N)^2}\,.
\end{align*}
So the probability that there exist $x,y\in \T_N$ such that  $\{x\stackrel{\sG_N }{\leftrightarrow}y\}$ but $\{x\stackrel{G_N }{\not\leftrightarrow}y\}$ is smaller than $N^2\e^{-(\log N)^2}$. This is a summable quantity in $N$, so the first Borel-Cantelli lemma ensures that this is not going to happen for $N$ large enough.
 
\smallskip

For equation \eqref{expecteddegreesimplified} in~\ref{coscc:iii} we simply calculate
\begin{align}
\E[\sD_1\,|\,W_1]
&=\E\Big[\sum_{x=2}^{N}\1{1\leftrightarrow x\mbox { {\small in} }\sG_N}\,\big|\,W_1\Big]
=(N-1)\P\big(W_2>N^\alpha (\log N)^2 W_1^{-1}\,|\,W_1\big)\nonumber
\end{align}
which immediately gives the desired formula. For \eqref{expecteddegreeoriginal} we first bound (recall that nearest neighbours are always connected)
\begin{align*}
\E[D_1\,|\,W_1]-2
	&=\sum_{y=3}^{N-1}\int_1^\infty 1-\e^{-wW_1\|1-y\|^{-\alpha}}{c\,}w^{-\tau}{\rm d}w\\
	&\leq \sum_{y=3}^{N-1}\Bigg({c\,}\frac{W_1}{\|1-y\|^\alpha}\int_1^{\|1-y\|^\alpha W_1^{-1}\vee 1} w^{1-\tau}{\rm d}w
	+\P\big(W_y>\|1-y\|^\alpha W_1^{-1}\vee 1\big)\Bigg)
\end{align*}
where we have splitted in two the integral and have upper bounded the integrand of the first part with \eqref{exp_ineq} and the integrand of the second part by $1$. A simple calculation shows that, for a fixed $y$, both the summands in the last brackets are bounded by a constant times $\|1-y\|^{-\gamma}W_1^{\tau-1}$, and summing over all $y$'s gives \eqref{expecteddegreeoriginal}.

\smallskip

For~\ref{coscc:iv}, formulas \eqref{degreeconcentrationoriginal} and \eqref{degreeconcentrationsimplified} can be proved in a very similar way, so we just show the first one. Abbreviate $U_N(x):={ \E[D_x\,|\,W_x]}^{ 1/2} \log N$ and use a union bound to get
\begin{align}
\P\big(\exists x\in\T_N:\,	 \big | D_x - \E[D_x\,|\,W_x]\big| > U_N(x)\big)
&\leq N\,\E\Big[\P\Big( \big | D_1 - \E[D_1\,|\,W_1]\big| > U_N(1) \,\big|\,W_1\Big)\Big]\label{eq:previous_U_N}
\end{align}
We observe that $D_1$ is the sum over $x=2,\dots, N$ of the variables $Z_{N,x}:=\mathbbm{1}_{\{1\leftrightarrow x \}}\,$,
which under  $\P(\,\cdot\,|\,W_1)$ are just $N-1$ independent Bernoullis. 
Since by~\ref{coscc:iii}
\begin{align*}
\sum_{x\neq 1}\E[Z_{N,x}^2\,|\,W_x]
=\E[D_x\,|\,W_x]\geq N^{\varepsilon}
\end{align*}
for some $0<\varepsilon<1-\gamma$, we can apply \eqref{dun} in Proposition \ref{prop:beb}  with $A_N=1$ to bound~\eqref{eq:previous_U_N}. 
We obtain 
\begin{align*}
\P\big(\exists x\in\T_N:\,	 \big | D_x - \E[D_x\,|\,W_x]\big| > U_N(x)\big)
	\leq N\e^{-c(\log N)^2}
\end{align*}
which is summable in $N$ and allows us to use the first Borel-Cantelli lemma to conclude.

\smallskip

For the second part of~\ref{coscc:v} we just integrate \eqref{expecteddegreesimplified}:
\begin{align*}
\E[\sD_{\sG_N}]
	&=N\E\Big[\E[\sD_1\,|\,W_1]\Big]
	=N\int_1^{\infty} \E[\sD_1\,|\,w] \,{c\,}w^{-\tau}\,{\rm d}w\\
	&={c}(N-1)N^{1-\gamma}(\log N)^{-2(\tau-1)}\int_1^{N^\alpha(\log N)^2}w^{-1}{\rm d}w +N(N-1)\P\big(W_1>N^\alpha(\log N)^2\big)
\end{align*}
and \eqref{ashp} follows.
We go back to the first equation of~\ref{coscc:v}. Using~\ref{coscc:iv} yields, $\P$-almost surely, for $N$ large enough,
\begin{align}\label{sciab}
\sD_{\sG_N}
	=\sum_{x\in\T_N}\sD_x
	&\lesssim \sum_{x\in\T_N} \E[\sD_x\,|\,W_x]+\sum_{x\in\T_N}{ \E[\sD_x\,|\,W_x]}^{1/2}.
\end{align}
We would like to invoke Proposition \ref{prop:beb} with $Z_{N,x}=\E[\sD_x\,|\,W_x]$, which are mutually independent under $\P$, and $A_N=N$. Condition \eqref{beb} is satisfied since, using item \ref{coscc:ii}, for some $0<\varepsilon<1-\gamma$
\begin{align*}
\sum_{x\in\T_N}\E[Z_{N,x}^2]
\geq N \E[(N-1)^2\1{W_1>N^\alpha(\log N)^2}]
\geq N^{2+\varepsilon}
\end{align*}
which is larger than $A_N^2(\log N)^2=N^2(\log N)^2$. Hence we get that, for $N$ large enough,
\begin{align*}
\sum_{x\in\T_N} \E[\sD_x\,|\,W_x]
	\lesssim \E[\sD_{\sG_N}]+{N^{1/2}\E\big[\E[\sD_1\,|\,W_1]^2\big]^{1/2}}.
\end{align*}
With \eqref{expecteddegreesimplified} at hand, we can calculate
\begin{align*}
\E\big[\E[\sD_0\,|\,W_0]^2\big]
	\lesssim {N^{2-2\gamma}}\int_1^{N^\alpha(\log N)^2}w^{2(\tau-1)}{c\,}w^{-\tau}\,{\rm d}w+N^2\P(W_1>N^\alpha(\log N)^2)
	\lesssim N^{2-\gamma}\,,
\end{align*}
thus
\begin{align}\label{UBoundSum}
\sum_{x\in\T_N} \E[\sD_x\,|\,W_x]
	\lesssim  \E[\sD_{\sG_N}]+ N^{ {(3-\gamma)}/2}\,.
\end{align}
We are left to deal with the last summand in \eqref{sciab}. We apply once more Proposition \ref{prop:beb}, this time with $Z_{N,x}=\E[\sD_x\,|\,W_x]^{1/2}$ and $A_N=N^{1/2}$. Condition \eqref{beb} is satisfied since $\sum_{x\in\T_N}\E[Z_{N,x}^2]=\E[\sD_{\sG_N}]$
which is larger than $A_N^2(\log N)^2=N(\log N)^2$ by \eqref{ashp}.  Therefore
\begin{align}\label{mari}
\sum_{x\in\T_N}{ \E[\sD_x\,|\,W_x]}^{1/2}
	\lesssim  N\E\big[{\E[\sD_1\,|\,W_1]}^{1/2}\big]+{\E[\sD_{\sG_N}]}^{1/2}
\end{align}
As before we bound
\begin{align*}
\E\big[{\E[\sD_1\,|\,W_1]}^{1/2}\big]
	\lesssim N^{\tfrac{1-\gamma}{2}}\int_1^{N^\alpha(\log N)^2}w^{\tfrac{\tau-1}{2}}{c\,}w^{-\tau}\,{\rm d}w+N^{ 1/2}\P(W_1>N^\alpha(\log N)^2)
	\lesssim N^{\tfrac{1-\gamma}{2}}\,.
\end{align*}
Putting this last result back into \eqref{mari} and combining it together with \eqref{UBoundSum} into \eqref{sciab}, gives an upper bound of the desired order for $\sD_{\sG_N}$, for $N$ large enough. A lower bound can be obtained in a completely similar way, yielding \eqref{zmb}.

\subsection{Proof of Proposition \ref{prop:VjConcentra}}\label{dib} Recall the notation $Q=(\log N)^{\tau -1}$.
For item \ref{trky} we first compute, for $j=1,\dots,\jmax-1$,
\begin{align*}
\E[|V_j|]
	&= N\,\P(W_1\in[N^{\alpha/2}(\log N)^{j}, N^{\alpha/2}(\log N)^{j+1})\\
	&= N^{1-\gamma/2}Q^{-j}\big(1-Q^{-1}\big)
\end{align*}
and analogously
\begin{align*}
\E[|V_{j^c}|]=N^{1-\gamma/2}Q^{j-2}\big(1-Q^{-1}\big)\,,\qquad
\E[|V_{j}^+|]=N^{1-\gamma/2}Q^{-j-1},\qquad
\E[|V_{j^c}^+|]=N^{1-\gamma/2}Q^{j-3}.
\end{align*}
Now we claim that, for all $j=1,\,\ldots,\,\jmax-1$ and for all $A\in\{V_j,V_{j^c},V_j^+,V_{j^c}^+\}$, $\P$--a.s.~for $N$ large enough,
\begin{align}
&\big||A|-\E[|A|]\big|\leq { \E[|A|] }^{1/2}\log N\,.\label{rgr}
\end{align}
Writing
%
 $|A|:=\sum_{x=1}^{N}\mathbbm{1}_{x\in A}$, we apply Proposition \ref{prop:beb} with $Z_{N,x}:=\mathbbm{1}_{x\in A}$ and $A_N=1$. Condition \eqref{beb} is satisfied since since
$$
\sum_{x=1}^{N}\E[Z_{N,x}^2]=\E[|A|]\geq (\log N)^2
$$ 
for all choices of $j$ and $A$, as can be easily checked. 
It follows that
\begin{align*}
\P(\mbox{For some $j$, }\exists A\in\{V_j,V_{j^c},V_j^+,V_{j^c}^+\}:\,|A-\E[|A|]|\geq \E[|A|]^{1/2}\log N) 
\leq 4\jmax \e^{-c(\log N)^2} 
\end{align*}
which is summable in $N$, and by the first Borel-Cantelli lemma \eqref{rgr} follows. To conclude, we put together \eqref{rgr} and the formulas for the expectations of $V_j,V_{j^c},V_j^+,V_{j^c}^+$. The case $j=\jmax$ can be treated in the very same way.

\smallskip

For item~\ref{via}, we just prove that $\sD_{V_j}>2\sD_{V_j^+}$ for any $j=1,\dots,\jmax-1$, since the proof for $\sD_{V_{j^c}}, \sD_{V_{j^c}^+}$ is very similar.
From \eqref{degreeconcentrationsimplified} we know that
\begin{align}
\sD_{V_j}
&=\sum_{x\in\T_N}\sD_x\mathbbm{1}_{x\in V_j}\geq \frac 12 \sum_{x\in\T_N}\E[\sD_x|W_x]\mathbbm{1}_{x\in V_j}\label{jwb}\\
\sD_{V_j^+}
&=\sum_{x\in\T_N}\sD_x \mathbbm{1}_{x\in V_j^+}\leq 2 \sum_{x\in\T_N}\E[\sD_x|W_x]\mathbbm{1}_{x\in V_j^+}\,.\label{itare}
\end{align}
almost surely for $N$ sufficiently large. Note that we neglect the error~$E[\overline D_x|W_x]^{1/2}$ appearing in~\eqref{degreeconcentrationsimplified} thanks to~\eqref{expecteddegreesimplified}.
For the first expression we use Proposition \ref{prop:beb} with $Z_{N,x}=\E[\sD_x|W_x]\mathbbm{1}_{x\in V_j}$. Using  \eqref{expecteddegreesimplified} one can check that, for $N$ large enough, $Z_{N,x}\leq A_N:=N^{1-{\gamma}/{2}}Q^{j}$.  Condition \eqref{beb} is verified since 
$$
\sum_{x=1}^N\E[Z_{N,x}^2]
	\geq N\, \E\Big[\E\big[\sD_1\,|\,W_1=N^{\alpha/2}(\log N)^{j}\big]^2\mathbbm{1}_{1\in V_j}\Big]
	\stackrel{\eqref{expecteddegreesimplified}}{\gtrsim} N^{3-\tfrac 32\gamma}Q^j
$$ 
which is larger than $A_N^2\log N$ 
for all  $j< \jmax$. Since
\begin{align}\label{blcla}
\E[Z_N]
	=\sum_{x=1}^N\E\Big[\E\big[\sD_x\,|\,W_x\big]\mathbbm{1}_{x\in V_j}\Big]
	&={c} N\int_{N^{\alpha/2}(\log N)^j}^{N^{\alpha/2}(\log N)^{j+1}}\E[\sD_1\,|\,W_1=w]\,w^{-\tau}\,{\rm d}w \nonumber\\
	&\stackrel{\eqref{expecteddegreesimplified}}{=}{c}  N^{1-\gamma}(N-1)Q^{-2}\log\log N\,,
\end{align}
which is much larger than $\sqrt{\sum_{x=1}^N\E[Z_{N,x}^2]}$,
\eqref{newyork} implies then that there exists some constant $c_1>0$ such that, $\P$--a.s.~for all $N$ large enough,
\begin{align}\label{pnlno}
\sD_{V_j}
	\stackrel{\eqref{jwb}}{\geq} \frac 12 Z_N\geq c_1\, N^{2-\gamma}Q^{-2}\log\log N\,.
\end{align}
Analogously, for the second expression in \eqref{itare} we use Proposition \ref{prop:beb} by posing $Z_{N,x}=\E[\sD_x|W_x]\mathbbm{1}_{x\in V_j^+}$. Condition \eqref{beb} is again verified since $\sum\E[Z_{N,x}^2]$ and 
$A_N^2\log N$ have the same orders as before.  So, by \eqref{newyork} and calculating $\E[\sD_{V_j^+}]$ as in \eqref{blcla},
\begin{align*}
\sD_{V_j^+}
	\leq 4 \E[\sD_{V_j^+}]
	\leq c_2\, N^{2-\gamma}Q^{-2}
\end{align*} 
for some $c_2>0$, $\P$--a.s.~for $N$ large enough. This together with \eqref{pnlno} concludes the argument.

\section*{Acknowledgments}

The authors would like to thank Noam Berger and Alexandre Stauffer for useful discussions. MS is also indebted to Quentin Berger, black belt of regularly varying functions, for helpful feedback and suggesting the proof of Lemma~\ref{okaramata}.

\smallskip

AC was supported by grant 613.009.102 of the Netherlands Organisation for Scientific Research (NWO), and MS by the French Research Agency (ANR) project ANR-16-CE32-0007-01 (CADENCE) and the MIUR
Departments of Excellence Program Math@Tov, CUP E83C18000100006.
MS acknowledges the hospitality of TU Delft where part of this work was carried out.


\bibliography{biblioCipSal}
\bibliographystyle{abbrvnat}

\end{document}